\documentclass[3p]{elsarticle}
\usepackage{graphicx} 
\usepackage{microtype} 
\usepackage{times}
\usepackage{amsmath,amsthm,amssymb}
\usepackage{epstopdf}
\usepackage{makecell}
\biboptions{numbers,sort&compress}
\usepackage{multirow}
\usepackage{booktabs}
\usepackage[colorlinks,linkcolor=blue,anchorcolor=blue,citecolor=blue]{hyperref}

\newtheorem{lemma}{Lemma}
\newtheorem{theorem}{Theorem}
\newtheorem{proposition}{Proposition}

\newtheorem{remark}{Remark}
\newtheorem{example}{Example}
\numberwithin{theorem}{section} 
\numberwithin{lemma}{section} 
\numberwithin{equation}{section}
\numberwithin{example}{section}
\numberwithin{table}{section}

\def\dfrac{\displaystyle\frac}

\journal{Journal}

\begin{document}
\begin{frontmatter}

\title{An optimal preconditioner for high-order scheme arising from multi-dimensional Riesz space fractional diffusion equations with variable coefficients \tnoteref{mytitlenote}}
\tnotetext[mytitlenote]{The work of Wei Qu was supported by the research grant 2024KTSCX069 from the Characteristic Innovation Projects of Ordinary Colleges and
Universities in Guangdong Province. The work of Sean Y. Hon was supported in part by NSFC under grant 12401544 and a start-up grant from the Croucher Foundation. The work of Siu-Long Lei was supported by the research grants MYRG-GRG2024-00237-FST and MYRG-GRG2023-00181-FST-UMDF from University of Macau.}

\author[add1]{Yuan-Yuan Huang}

\author[add2]{Wei Qu}

\author[add1]{Sean Y. Hon}

\author[add3]{Siu-Long Lei\corref{cor1}}\ead{sllei@um.edu.mo}
\cortext[cor1]{Corresponding author}

\address[add1]{Department of Mathematics, Hong Kong Baptist University, Kowloon Tong, Hong Kong Special Administrative Region of China}
\address[add2]{School of Mathematics and Statistics, Shaoguan University, Shaoguan 512005, China}
\address[add3]{Department of Mathematics, University of Macau, Macao Special Administrative Region of China}

\begin{abstract}
In this paper, we propose an efficient method for solving multi-dimensional Riesz space fractional diffusion equations with variable coefficients. The Crank–Nicolson (CN) method is used for temporal discretization, while the fourth-order fractional centered difference (4FCD) method is employed for spatial discretization. Using a novel technique, we show that the CN-4FCD scheme for the multi-dimensional case is unconditionally stable and convergent, achieving second-order accuracy in time and fourth-order accuracy in space with respect to the discrete $L^2$-norm. Moreover, leveraging the symmetric multi-level Toeplitz-like structure of the coefficient matrix in the discrete linear systems, we enhance the computational efficiency of the proposed scheme with a sine transform-based preconditioner, ensuring a mesh-size-independent convergence rate for the conjugate gradient method. Finally, two numerical examples validate the theoretical analysis and demonstrate the superior performance of the proposed preconditioner compared to existing methods.
\end{abstract}

\begin{keyword}
Variable-coefficient Riesz space fractional diffusion equations; High-order symmetric multi-level Toeplitz systems; Stability and convergence; Sine transform based preconditioner; Mesh-sizes independent convergence rate; Krylov subspace methods
\end{keyword}

\end{frontmatter}

\section{Introduction}
Over the past few decades, fractional differential equations have been widely studied due to their numerous applications and their enhanced modeling capabilities across a broad spectrum of problems, such as, physics \cite{BWM1,BWM2}, biology \cite{Magin1}, image processing \cite{BF1}, finance \cite{RSM1} and more. 

Compared to integer differential operators, fractional operators exhibit a non-local feature, which brings significant challenges to research. In most cases, it is even impossible to obtain the exact solutions of fractional differential equations. For this reason, more attention is being devoted to solving fractional differential equations efficiently using numerical methods. Among them, we mention some widely used numerical methods such as finite difference methods \cite{MT1,MT2,SL1,TZD1,lin2021stability,LL1,she2022class,zhu2021efficient,huang2023tau}, finite volume methods \cite{fu2019stability,donatelli2018spectral,LZTBA1,fu2019finite,qu2021fast}, finite element methods \cite{JM1,JLPZ1,bu2014galerkin}. 

However, solving the discrete linear systems generated by the above methods is computationally demanding, primarily because fractional differential operators are non-local. As a result, their numerical discretization leads to dense linear systems, which are usually ill-conditioned as the matrix size or the dimension increases. Consequently, this has promoted the development of fast algorithms for solving such discrete linear systems. Various techniques have been proposed and analyzed in the literature. Representative methods include the multigrid method \cite{PS1,moghaderi2017spectral,lin2017multigrid}, Krylov subspace methods with circulant preconditioners \cite{lei2013circulant,lei2016multilevel}, banded preconditioners \cite{JLZ1,SLYL1}, approximate inverse preconditioners \cite{PKNS1}, structure-preserving preconditioners \cite{DMS1}, rational preconditioners \cite{aceto2023rational}, and sine transform based preconditioners \cite{huang2021spectral,zhang2022fast,huang2023tau,hon2024symbol,huang2024optimal,li2025multilevel,huang2025efficient}. In addition to these methods, there exist many other fast solution algorithms for fractional differential equations, which will not be discussed here.

In this paper, we focus on the following multi-dimensional Riesz space fractional diffusion equations (RSFDEs) with variable coefficients \cite{bai2020fast,zhang2024two,she2024tau}:
\begin{equation}\label{RSFDEs}
\begin{small}
\left\{
\begin{aligned}
&r(\boldsymbol{x},t)\frac{\partial u(\boldsymbol{x},t)}{\partial t} = \sum_{i=1}^d K_i \frac{\partial^{\alpha_i}u({\boldsymbol{x}},t)}{\partial\left|x_i\right|^{\alpha_i}}+f(\boldsymbol{x},t),~~\boldsymbol{x}\in\Omega\subset \mathbb{R}^d,~~t\in(0,T],\\
&u(\boldsymbol{x},t)=0, \quad \boldsymbol{x}\in\partial\Omega,\\
&u(\boldsymbol{x},0)=\psi(\boldsymbol{x}),~\boldsymbol{x}\in\Omega,
\end{aligned}
\right.
\end{small}
\end{equation}
where $\alpha_i \in (1,2)$, $K_i>0$ for $i=1,2,\ldots,d$ with $d\geq 1$ being the dimension of the problem. $r(\boldsymbol{x},t)$ is a bounded positive
function, i.e., $0<\check{r}\leq r({\boldsymbol{x}}, t)\leq {\hat r}$. $\Omega=\prod\limits_{i=1}^d\left(\check{a}_i, \hat{a}_i\right)$ and $\partial\Omega$ denotes its boundary; $\boldsymbol{x}=\left(x_1, x_2, \ldots, x_d\right)$ is a point in $\mathbb{R}^d$. $\psi(\boldsymbol{x})$ is the initial value and $f(\boldsymbol{x},t)$ is a given source term. $\frac{\partial^{\alpha_i}u({\boldsymbol{x}},t)}{\partial\left|x_i\right|^{\alpha_i}}$ denotes the Riesz fractional derivative with respect to the variable $x_i$,
which is defined by
\begin{equation}\nonumber
\begin{small}
\frac{\partial^{\alpha_i}u({\boldsymbol{x}},t)}{\partial\left|x_i\right|^{\alpha_i}}=-\frac{1}{2\cos(\alpha_i\pi/2)\Gamma(2-\alpha_i)} \frac{\partial^{2}}{\partial {x_i^2}}\int_{{\check a}_i}^{\hat{a}_i}\frac{u(x_1,x_2,\ldots,x_{i-1},\eta,{x_{i+1}},\dots,x_d,t)}{|x_i-\eta|^{\alpha_i-1}}{\mathrm{d}}\eta,
\end{small}
\end{equation}
in which $\Gamma(\cdot)$ denotes the gamma function.

To date, the model (\ref{RSFDEs}) has attracted significant attention from researchers and several efficient numerical schemes have been developed to solve (\ref{RSFDEs}). More specifically, one of the most widely used schemes is the first-order scheme, which is based on the relationship between the Riesz derivative and the left and right Riemann-Liouville (RL) derivatives. For example, Bai et al. \cite{bai2017diagonal,bai2020fast} adopted the shifted Grünwald-Letnikov formula proposed by Meerschaert and Tadjeran \cite{MT2} to obtain a first-order scheme for $d$-dimensional problems ($d=1,2,3$). After discretization using the finite difference method, a symmetric positive definite (SPD) diagonal-plus-Toeplitz linear system was obtained for (\ref{RSFDEs}), whose coefficient matrix at each time level has the form of $D+T$ with $D$ being a positive diagonal matrix and $T$ being a symmetric $d$-level Toeplitz matrix. 

In the early stage, for solving the diagonal-plus-Toeplitz system arising from the image deconvolution problem, Ng and Pan \cite{ng2010approximate} proposed the preconditioned conjugate gradient (PCG) method using an approximate inverse circulant-plus-diagonal preconditioner. This approach can be implemented efficiently using fast Fourier transforms and avoids the direct inversion of the coefficient matrices. 
Later, as another application for RSFDEs with variable coefficients, Bai et al. \cite{bai2017diagonal} proposed an unconditionally convergent diagonal and Toeplitz splitting (DTS) iteration method for solving such systems in the one-dimensional case. In addition, for further improving the convergence of Krylov subspace iteration methods like the generalized minimal residual (GMRES) method and the biconjugate gradient stabilized (BiCGSTAB) method, the authors also constructed a diagonal and circulant splitting (DCS) preconditioner, which is induced from the DTS iteration method with a circulant or skew-circulant approximation to the associated Toeplitz matrix,
see \cite{bai2017diagonal,Lu-CAM} for more details. Extending the idea in \cite{bai2017diagonal}, Bai and Lu \cite{bai2020fast} also proposed the DCS-type preconditioners for solving two- and three-dimensional cases. Numerical results are reported to show that the fast convergence rate of the preconditioned GMRES (PGMRES) method and preconditioned BiCGSTAB (PBiCGSTAB) method with the proposed preconditioners.



As shown in \cite{bini1990new,serra1999superlinear}, the natural $\tau$ and optimal $\tau$ matrices, which can be diagonalized by the sine transform matrix, are good approximation to the symmetric Toeplitz matrices. Inspired by this fact, Huang et al. in \cite{huang2021spectral} proposed a sine transform based preconditioner for solving the ill-conditioned symmetric Toeplitz linear system that arises from steady-state multi-dimensional Riesz fractional diffusion equation with constant coefficients. They also proved that the spectrum of the preconditioned matrix involving the natural $\tau$ preconditioner is uniformly bounded in the open interval $(1/2,3/2)$. As a result, the PCG method with the $\tau$ preconditioner converges independently of the mesh size, outperforming that with the circulant preconditioner from both a theoretical and numerical perspective. 

Following the idea in \cite{huang2021spectral}, various sine transform based splitting preconditioners induced from the DTS iteration method have been designed for the discrete linear system; see, for instance, \cite{lu2021splitting,Tang2022AML,shao2022preconditioner,tang2024new} and references therein. Moreover, by approximating the inverse of SPD Toeplitz matrix with a sine transform based matrix and combining them row-by-row, Zeng el al. \cite{zeng2022tau} proposed a novel sine transform based approximate inverse preconditioning technique for diagonal-plus-Toeplitz systems, and the performance of the proposed preconditioner is satisfactory. However, it is remarked that the induced preconditioners and the approximate inverse preconditioner are not symmetric, hence, the PCG method cannot be chosen for solving the resulting SPD linear systems. Instead, the authors resort to the PGMRES method. Although the PGMRES method achieves a rapid convergence rate from numerical experiments, theoretical results showing that the spectrum of the preconditioned matrix is clustered around 1. However, in general, the eigenvalues alone cannot fully characterize the convergence rate of GMRES \cite{greenbaum1996any,greenbaum1994matrices}. Thus, it is better to prove the mesh-size independent convergence of the PGMRES method. Unfortunately, at this stage, this issue remains unresolved even for the first-order scheme. 


To improve spatial discretization accuracy, various second-order methods \cite{TZD1,lin2021stability,SL1,LL1} have been subsequently developed to discretize the RL derivatives, which are essentially suitable for solving such RSFDEs (\ref{RSFDEs}). In addition, by introducing the concept of the fractional central difference (FCD) operator given by Ortigueira in \cite{ortigueira2006riesz}, \c{C}elik and Duman \cite{CD2} constructed a second-order FCD (2FCD) scheme to solving RSFDEs with constant coefficients. Very recently, Zhang et al. \cite{zhang2024two,zhang2025two} have respectively adopted the second-order FCD scheme for solving high-dimensional RSFDEs (\ref{RSFDEs}) and variable-coefficient fractional diffusion equations with time delay problem. Meanwhile, they also proposed the PCG method with $\tau$ preconditioner for solving the symmetric $D+T$ linear systems, and the convergence of the PCG method is also shown to be mesh-sizes independent. Besides, by adopting the quasi-compact operator proposed by Hao et al. \cite{hao2015fourth}, a novel quasi-compact scheme for high-dimensional RSFDEs (\ref{RSFDEs}) has also been constructed in \cite{zhang2024two}, which further improve the spatial accuracy to fourth order. 

Typically, the quasi-compact operator generates a tridiagonal matrix, which consequently results in a non-symmetric discrete linear system for variable-coefficient RSFDEs (\ref{RSFDEs}). In this case, Zhang et al. \cite{zhang2024two} also utilize the PBiCGSTAB method with a sine transform based preconditioner for solving the non-symmetric linear systems. However, the spectral analysis for the preconditioned matrix is not derived theoretically. Later, a one-sided PGMRES solver for the quasi-compact scheme is studied by She et al. \cite{she2024tau}. More importantly, the mesh-size independent convergence of the PGMRES method with the $\tau$ preconditioner is analyzed accordingly. To the best of our knowledge, this may be the first attempt to present a rigorous theoretical analysis of the convergence of the PGMRES method for solving the variable-coefficient fractional diffusion equations, which indeed guarantees the efficiency of the PGMRES method. 

To balance computational efficiency and high-order accuracy for solving (\ref{RSFDEs}), a novel fourth-order scheme is proposed in this paper. Extending beyond existing approaches, we make the following contributions: 
\begin{itemize}
    \item Using a new technique, we establish the stability and convergence of the proposed scheme for the $d$-dimensional case without any restrictions on the spatial step size. In contrast, the stability and convergence of the quasi-compact scheme is analyzed under the assumption that the spatial step is small. For more details, see Theorems 3.1 and 3.2 in \cite{zhang2024two}.
    
    
    
    \item Compared to the second-order 2FCD scheme \cite{CD2}, our proposed scheme achieves fourth-order spatial accuracy while maintaining the same matrix structure as the 2FCD scheme, without incurring additional computational cost or storage requirements. Furthermore, in contrast to the fourth-order quasi-compact scheme \cite{zhang2024two,she2024tau}, our approach attains comparable fourth-order accuracy while avoiding the introduction of tridiagonal matrices. Importantly, our scheme preserves the symmetric structure of the resulting coefficient matrix, which simplifies implementation when using the PCG method.
    
    \item We develop a novel sine transform based preconditioner to efficiently solve the linear systems arising from the RSFDEs (\ref{RSFDEs}), and show that the spectrum of the preconditioned matrix is uniformly bounded within an interval independent of the mesh size; see Theorem \ref{eigenprematrix}. This result implies that the PCG method, when equipped with the proposed preconditioner, can achieve mesh-size independent convergence rate. To the best of our knowledge, this is the first theoretical analysis of the convergence behavior of the PCG method for high-order schemes associated with variable-coefficient RSFDEs. Our numerical experiments support the theoretical findings and further demonstrate that the proposed method outperforms existing preconditioned Krylov subspace methods in terms of computational efficiency.
\end{itemize}

The rest of this paper is organized as follows. In Section $2$, we present the fully discrete scheme for multi-dimensional RSFDEs (\ref{RSFDEs}) with variable coefficients and establish its unconditional stability and convergence in the discrete $L^2$-norm. In Section $3$, by
exploiting the symmetric multilevel Toeplitz structure of the discrete linear system, we propose a sine transform based preconditioning strategy to accelerate the convergence rate of the CG method when solving the resulting linear systems and show the PCG method can achieve an optimal convergence rate (a convergence rate independent of mesh sizes). In Section $4$, we carry out numerical experiments to verify the accuracy of the proposed scheme and illustrate the good performance of the proposed preconditioned solver. Finally, we provide some concluding remarks in Section $5$.

\section{Fully-Discrete Scheme}\label{Sec-2}
For the fourth-order fractional centered difference
(4FCD) approximation of the operator $\frac{\partial^{\alpha_i}u({x})}{\partial\left|x\right|^{\alpha_i}}$, we have the following lemma:
\begin{lemma}{\rm (\cite{ding2023high})}\label{1Doperator}
Let 
\begin{equation}\nonumber
\mathcal{L}^{4+\mu}(\mathbb{R})=\left\{u(x)\in L^{1}(\mathbb{R})|\int_{-\infty}^{\infty}(1+|\eta|)^{4+\mu}|\hat{u}(\eta)|{\mathrm d}\eta<\infty\right\},
\end{equation}
be the fractional Sobolev space, where $\hat{u}(\eta)=\int_{-\infty}^{\infty}u(x)e^{{\bf i}\eta x}dx$ is the Fourier transformation of $u(x)$ for all $\eta\in(-\infty,\infty)$ and ${\bf i}=\sqrt{-1}$. 
Suppose $u(x)\in\mathcal{L}^{4+\mu}(\mathbb{R})$. Then, we have the following fourth-order approximation:
\begin{eqnarray}\nonumber
\dfrac{\partial^{\alpha_i} u(x)}{\partial |x|^{\alpha_i}}=\delta_x^{\alpha_i}u(x)+\mathcal{O}(h^4), \quad h\to 0,
\end{eqnarray}
where $h$ is the spatial step size and the operator $\delta_{x}^{\alpha_i}u(x)$ is defined by
\begin{eqnarray}\label{4FCD}
\delta_{x}^{\alpha_i}u(x)=-\frac{1}{h^{\alpha_i}}\sum_{k=-\infty}^{\infty}s_{k}^{(\alpha_i)}
u(x-kh).
\end{eqnarray}
Here 
the coefficients $s_{k}^{(\alpha_i)}$ $(k=0,\pm 1,\pm 2,\ldots)$ are determined by the Fourier expansion of the generating function
\begin{equation}\label{generatingS}
S^{(\alpha_i)}(\omega):=\left[1+\frac{\alpha_i}{24}(2-\omega-\omega^{-1})\right](2-\omega-\omega^{-1})^\frac{\alpha_i}{2}=\sum_{k=-\infty}^{\infty}s_{k}^{(\alpha_i)}\omega^k,\quad |\omega|\leq 1.
\end{equation}            
Setting $\omega=e^{{\bf i}\theta}$ with ${\bf i}=\sqrt{-1}$ in Equation (\ref{generatingS}) and using the inverse Fourier transform formula coupling the following formula,
\begin{equation}\label{Riesz_coef}
\frac{1}{2\pi}\int_{-\pi}^{\pi}(2-e^{{\bf i}\theta}-e^{{\bf -i}\theta})^{\frac{\alpha_i}{2}e^{-{\bf i} k\theta}}d\theta=\frac{(-1)^{k}\Gamma(\alpha_i+1)}{\Gamma(\frac{\alpha_i}{2}-k+1)\Gamma(\frac{\alpha_i}{2}+k+1)},
\end{equation}
one gets the explicit expression of the coefficient $s_{k}^{(\alpha_i)}$,
\begin{equation}\label{coesk}
s_{k}^{(\alpha_i)}=\frac{(-1)^{k}\Gamma(\alpha_i+1)}{\Gamma(\frac{\alpha_i}{2}-k+1)\Gamma(\frac{\alpha_i}{2}+k+1)}
\left[1+\frac{\alpha_i(\alpha_i+1)(\alpha_i+2)}{6(\alpha_i-2k+2)(\alpha_i+2k+2)}\right].
\end{equation}
\end{lemma}
It is noted that the coefficients $s_{k}^{(\alpha_i)}$ satisfy the following properties:
\begin{proposition}{\rm (\cite{ding2023high})}\label{coef_prop}
For all $1<\alpha_i<2$, and let $s_{k}^{(\alpha_i)}$ be defined as (\ref{coesk}). Then, we have
\begin{description}
  \item[(i)]  $s_{k}^{(\alpha_i)}=s_{-k}^{(\alpha_i)}$ \mbox{for all} $k\geq 1$;
  \item[(ii)]  $s_{0}^{(\alpha_i)}>0$,\,$s_{\pm 1}^{(\alpha_i)}<0$;
   \item[(iii)]
   $s_{\pm 2}^{(\alpha_i)}\left\{
                               \begin{array}{ll}
                                 \leq 0, \,& \hbox{$\alpha\in(1,\alpha_i^{*}]$,} \\
                                 \geq 0, \,& \hbox{$\alpha\in[\alpha_i^{*},2)$,}
                               \end{array}
                             \right.$
   where $\alpha_i^{*}\approx 1.6516$;
    \item[(iv)] $s_{\pm k}^{(\alpha_i)}\leq0$, $k=3, 4,\ldots$; \mbox{and}
    \item[(v)]  $\sum\limits_{k=-\infty}^{\infty}s_{k}^{(\alpha_i)}=0.$
\end{description}
\end{proposition}

To derive the discretization for (\ref{RSFDEs}), we first introduce some basic notation, which is necessary to carry out the theoretical analysis. For any nonnegative integer $m$, $n$ with $m\leq n$, define the set $m\wedge n:=\{m,m+1,...,n-1,n\}$. 
Denote
\begin{equation*}
	m_1^{-}=m_d^{+}=1,\quad m_i^{-}=\prod\limits_{j=1}^{i-1}m_i,~i\in 2\wedge d,\quad  m_k^{+}=\prod\limits_{j=k+1}^{d}m_j,~k\in 1\wedge (d-1).
\end{equation*}

Let $\Delta t=T/M$ be the size of the time step with a positive integer $M$. We define the temporal partition $t_m=m\Delta t$ for $m=0\wedge M$.
Let $h_i=(\hat{a}_i-\check{a}_i)/(n_i+1)$ be the size of the spatial step 
with positive integers $n_i$. Let $\mathbb{Z}$ be the set of all integer numbers. For $i=1\wedge d$, define ${\mathbb{K}_i}=\{i\in{\mathbb{Z}|i=1\wedge n_i}\}$, ${\mathbb{\hat{K}}_i}=\{i\in{\mathbb{Z}|i=0\wedge n_i+1}\}$, ${\mathbb{L}}={\mathbb{K}_1}\times{\mathbb{K}_2}\times\cdots\times{\mathbb{K}_d}$, ${\mathbb{\hat{L}}}={\mathbb{\hat{K}}_1}\times{\mathbb{\hat{K}}_2}\times\cdots\times{\mathbb{\hat{K}}_d}$, $\partial{\mathbb{L}}={\mathbb{\hat{L}}}\setminus  {\mathbb{L}}$. Set a multi-index ${\boldsymbol J}=(j_1,j_2,\ldots,j_d)$, and denote by ${\boldsymbol x_J}=(x_{1,j_1},x_{2,j_2},\ldots,x_{d,j_d})=(\check{a}_1+j_1h_1,\check{a}_2+j_2h_2,\ldots,\check{a}_d+j_dh_d)$ the grid node. Then define $U_{\boldsymbol J}^m=U({\boldsymbol x_J},t_m)$, $r_{\boldsymbol J}^{m+\frac{1}{2}}=r({\boldsymbol x_J},t_{m+\frac{1}{2}})$ and $f_{\boldsymbol J}^{m+\frac{1}{2}}=f({\boldsymbol x_J},t_{m+\frac{1}{2}})$, respectively.

Based on Lemma \ref{1Doperator}, we can extend the 4FCD formula (\ref{4FCD}) to discretize the multi-dimensional Riesz fractional derivative. Hence, the discrete
formula of $\frac{\partial^{\alpha_i}u({\boldsymbol x_J},t)}{\partial\left|x_i\right|^{\alpha_i}}$with respect to the $i$-th variable $x_i$ at point $({\boldsymbol {x_J},t_{m+\frac{1}{2}}})$ is expressed as
\begin{equation}\label{spacediscretization}
\frac{\partial^{\alpha_i}U(\boldsymbol{x_J},t_{m+\frac{1}{2}})}{\partial\left|x_i\right|^{\alpha_i}}=\delta_{x_i}^{\alpha_i}U_{{\boldsymbol J}}^{m+\frac{1}{2}}+\mathcal{O}(h_i^4):=-\frac{1}{h_i^{\alpha_i}}\sum_{k=-\infty}^{\infty}s_{k}^{(\alpha_i)}
U_{\boldsymbol{J}_i-k}^{m+\frac{1}{2}}+\mathcal{O}(h_i^4),
\end{equation}
where ${\boldsymbol{J}_i-k}=(j_1,\ldots,j_{i-1},j_i-k,j_{i+1},\ldots,j_d)$ are the multi-index, and $s_{k}^{(\alpha_i)}\, (i=1\wedge d)$ satisfy the properties in Proposition \ref{coef_prop}. Next, applying (\ref{spacediscretization}) to Equation (\ref{RSFDEs}) at the point 
$({\boldsymbol x_J},t_{m+\frac{1}{2}})$, 
we have
\begin{equation}
r_{{\boldsymbol J}}^{m+\frac{1}{2}}\frac{\partial U_{{\boldsymbol J}}^{m+\frac{1}{2}}}{\partial t}=\sum_{i=1}^d K_i \delta_{x_i}^{\alpha_i}U_{{\boldsymbol J}}^{m+\frac{1}{2}}+f_{{\boldsymbol J}}^{m+\frac{1}{2}}+R_{{\boldsymbol J}}^{m+\frac{1}{2}},
\end{equation}
where there exists a positive constant $c_1$ such that
\begin{equation}
\Big|R_{{\boldsymbol J}}^{m+\frac{1}{2}}\Big| \leq c_1 (h_1^4+h_2^4+\cdots+h_d^4).
\end{equation}
Using $\frac{U_{{\boldsymbol J}}^{m+1}-U_{{\boldsymbol J}}^{m}}{\Delta t}$ to approximate $\frac{\partial U_{{\boldsymbol J}}^{m+\frac{1}{2}}}{\partial t}$, we can derive
\begin{equation}\label{numerical_scheme}
  r_{{\boldsymbol J}}^{m+\frac{1}{2}}\frac{U_{{\boldsymbol J}}^{m+1}-U_{{\boldsymbol J}}^{m}}{\Delta t}= \sum_{i=1}^d K_i \delta_{x_i}^{\alpha_i}U_{{\boldsymbol J}}^{m+\frac{1}{2}}+f_{{\boldsymbol J}}^{m+\frac{1}{2}}+R_{{\boldsymbol J}}^{m+\frac{1}{2}}+{\hat R}_{{\boldsymbol J}}^{m+\frac{1}{2}},
\end{equation}
where $\Big|{\hat R}_{{\boldsymbol J}}^{m+\frac{1}{2}}\Big|\leq c_2 \Delta t^2$ with $c_2$ being a positive constant. Omitting the small terms ${R}_{{\boldsymbol J}}^{m+\frac{1}{2}}$ and ${\hat R}_{{\boldsymbol J}}^{m+\frac{1}{2}}$, and replacing the exact solution $U_{{\boldsymbol J}}^{m+\frac{1} {2}}$ with its numerical solution $u_{{\boldsymbol J}}^{m+\frac{1} {2}}$, the CN-4FCD scheme for (\ref{RSFDEs}) can be written as 
\begin{equation}\label{4thscheme_operator}
  r_{{\boldsymbol J}}^{m+\frac{1}{2}}\frac{u_{{\boldsymbol J}}^{m+1}-u_{{\boldsymbol J}}^{m}}{\Delta t}=\sum_{i=1}^d K_i \delta_{x_i}^{\alpha_i}u_{{\boldsymbol J}}^{m+\frac{1}{2}}+f_{{\boldsymbol J}}^{m+\frac{1}{2}}.  
\end{equation}
In the following, we will study the stability and convergence of the CN-4FCD scheme. 
For simplicity, we define the vector consisting of spatial grid points with $j_1$-dominant ordering
\begin{equation}
{{\bf P}_d=\Big(P_{1,1,\ldots,1,},\ldots,P_{n_1,1,\ldots,1},\ldots,P_{1,n_2,\ldots,n_d},\ldots,P_{n_1,n_2,\ldots,n_d}\Big)^T},
\end{equation}
where $P_{j_1,j_2,\ldots,j_d}$ denotes the grid points $(x_{1,j_1},x_{2,j_2},\ldots,x_{d,j_d})$, $j_i=1\wedge n_i$. Let ${\bf u}^{m}\approx U({\bf P}_d,t_m)$, $D^{m+\frac{1}{2}}={\rm diag}(r({\bf P}_d,t_{m+\frac{1}{2}}))$, and ${\bf f}^{m+\frac{1}{2}}=f({\bf P}_d,t_{m+\frac{1}{2}})$. Then, the matrix form of the scheme (\ref{4thscheme_operator}) is written as 
\begin{equation}\label{matrix-form}
(D^{m+\frac{1}{2}}+T){\bf u}^{m+1}=(D^{m+\frac{1}{2}}-T){\bf u}^{m}+\Delta t {\bf f}^{m+\frac{1}{2}},\quad m=0\wedge M-1,
\end{equation}
where 
\begin{equation}\label{multilevel_T}
T=\sum_{i=1}^d \eta_i I_{n_i^{-}} \otimes S_{n_i}^{\left(\alpha_i\right)} \otimes I_{n_i^{+}}
\end{equation}
with
\begin{equation*}
\eta_i=\frac{K_i \Delta t}{2h_i^{\alpha_i}}>0
\end{equation*}
and
\begin{eqnarray}\label{Toeplitz-Walpha}
S_{n_i}^{\left(\alpha_i\right)}=\left[\begin{array}{ccccc}
s_0^{(\alpha_i)} & s_{-1}^{(\alpha_i)} & \cdots & s_{2-n_i}^{(\alpha_i)} & s_{1-n_i}^{(\alpha_i)}\\
s_{1}^{(\alpha_i)} & s_0^{(\alpha_i)} & s_{-1}^{(\alpha_i)} & \ddots & s_{2-n_i}^{(\alpha_i)}\\
\vdots & s_{1}^{(\alpha_i)} & s_0^{(\alpha_i)} & \ddots & \vdots\\
s_{n_i-2}^{(\alpha_i)} & \ddots & \ddots & \ddots & s_{-1}^{(\alpha_i)}\\
s_{n_i-1}^{(\alpha_i)} & s_{n_i-2}^{(\alpha_i)} & \cdots & s_{1}^{(\alpha_i)} & s_0^{(\alpha_i)}
\end{array}
\right].
\end{eqnarray}

It follows from \cite{ding2023high} that the generating function $S^{(\alpha_i)}(e^{{\bf i}\theta})$ of the Toeplitz matrix $S_{n_i}^{\left(\alpha_i\right)}$ is 
nonnegative for $\theta \in [-\pi, \pi]$ and $1<\alpha_i<2$. Obviously, $S^{(\alpha_i)}(e^{{\bf i}\theta})$ is not identically zero. By Grenander–Szegö Theorem \cite{chan2007introduction}, we can conclude that the matrix $S_{n_i}^{\left(\alpha_i\right)}$ is SPD, which, combining with the fact $\eta_i>0$ and the properties of the Kronecker product, implies that the matrix $T$ defined in (\ref{multilevel_T}) is also SPD.


Multiplying both sides of (\ref{matrix-form}) by $(D^{m+\frac{1}{2}})^{-1}$ gives
\begin{eqnarray}\label{matrix-formU}
\Big(I+(D^{m+\frac{1}{2}})^{-1}T\Big){\bf u}^{m+1}=\Big(I-(D^{m+\frac{1}{2}})^{-1}T\Big){\bf u}^{m}+\Delta t (D^{m+\frac{1}{2}})^{-1}{\bf f}^{m+\frac{1}{2}}.
\end{eqnarray}
Now, we proceed to establish the stability theorem as follows:
\begin{theorem}\label{stability theorem}
{\rm (Stability)} The finite difference scheme (\ref{matrix-formU}) is unconditionally stable.
\end{theorem}
\begin{proof}
Suppose ${\bf v}^{m+1}$ is a solution to the following difference scheme
\begin{eqnarray}\label{matrix-formV}
\Big(I+(D^{m+\frac{1}{2}})^{-1}T\Big){\bf v}^{m+1}=\Big(I-(D^{m+\frac{1}{2}})^{-1}T\Big){\bf v}^{m}+\Delta t (D^{m+\frac{1}{2}})^{-1}{\bf f}^{m+\frac{1}{2}}.
\end{eqnarray}
Denote 
$$
{\boldsymbol{\epsilon}}^{m+1}={\bf u}^{m+1}-{\bf v}^{m+1}.
$$
Subtracting (\ref{matrix-formV}) from (\ref{matrix-formU}) leads to the following perturbation system
\begin{equation}\label{matrix-formepsilon}
\Big(I+(D^{m+\frac{1}{2}})^{-1}T\Big){\boldsymbol{\epsilon}}^{m+1}
=\Big(I-(D^{m+\frac{1}{2}})^{-1}T\Big){\boldsymbol{\epsilon}}^{m},
\end{equation}
i.e.,
\begin{equation}\label{matrix-formepsilon2}
{\boldsymbol{\epsilon}}^{m+1}
=\Big(I+(D^{m+\frac{1}{2}})^{-1}T\Big)^{-1}\Big(I-(D^{m+\frac{1}{2}})^{-1}T\Big){\boldsymbol{\epsilon}}^{m}.
\end{equation}
Since 
\begin{equation}\label{similar_matrix}
(D^{m+\frac{1}{2}})^{\frac{1}{2}}\Big((D^{m+\frac{1}{2}})^{-1}T\Big)(D^{m+\frac{1}{2}})^{-\frac{1}{2}}=(D^{m+\frac{1}{2}})^{-\frac{1}{2}}T(D^{m+\frac{1}{2}})^{-\frac{1}{2}},
\end{equation}
we know that $(D^{m+\frac{1}{2}})^{-1}T$ and $(D^{m+\frac{1}{2}})^{-\frac{1}{2}}T(D^{m+\frac{1}{2}})^{-\frac{1}{2}}$ are similar. In addition, the matrix $T$ is SPD implies that $(D^{m+\frac{1}{2}})^{-\frac{1}{2}}T(D^{m+\frac{1}{2}})^{-\frac{1}{2}}$ is also SPD, which means that ${\lambda((D^{m+\frac{1}{2}})^{-1}T)}>0$ and 
\begin{equation}\label{eig_less_1}
\Big|\Big(1+\lambda((D^{m+\frac{1}{2}})^{-1}T)\Big)^{-1}\Big(1-\lambda((D^{m+\frac{1}{2}})^{-1}T)\Big)\Big|<1
\end{equation}
for any eigenvalue $\lambda$ of $(D^{m+\frac{1}{2}})^{-1}T$. Note that
$\Big(1+\lambda((D^{m+\frac{1}{2}})^{-1}T)\Big)^{-1}\Big(1-\lambda((D^{m+\frac{1}{2}})^{-1}T)\Big)$ is an eigenvalue of the matrix $\Big(I+(D^{m+\frac{1}{2}})^{-1}T\Big)^{-1}\Big(I-(D^{m+\frac{1}{2}})^{-1}T\Big)$. (\ref{eig_less_1}) indicates that the spectral radius of the matrix $\Big(I+(D^{m+\frac{1}{2}})^{-1}T\Big)^{-1}\Big(I-(D^{m+\frac{1}{2}})^{-1}T\Big)$ is less than 1. Therefore, the scheme (\ref{matrix-formU}) is unconditionally stable. The proof is complete.
\end{proof}

 Before discussing the convergence of the CN-4FCD scheme,
we first introduce the concepts of inner product and the norm of the grid functions in 
\begin{equation}
    U_h:=\Big\{{\bf u}|{\bf u}=\Big(u_{1,1,\ldots,1,},\ldots,u_{n_1,1,\ldots,1},\ldots,u_{1,n_2,\ldots,n_d},\ldots,u_{n_1,n_2,\ldots,n_d}\Big)^T\Big\}.
\end{equation}
For ${\bf u}, {\bf v}\in U_h$,  we define
\begin{equation}
   \langle {\bf u},{\bf v}\rangle:=\prod_{i=1}^d h_i\sum_{\boldsymbol J\in {\mathbb L}}{\bf u}_{\boldsymbol J}{\bf v}_{\boldsymbol J},
\end{equation}
in particular,
\begin{equation} ||{\bf u}||^2=\langle {\bf u},{\bf u}\rangle, \quad ||{\bf u}||:=\sqrt{\prod_{i=1}^d h_i\sum_{\boldsymbol J\in {\mathbb L}}|{\bf u}_{\boldsymbol J}|^2},\quad ||{\bf u}||_2:=\sqrt{\sum_{\boldsymbol J\in {\mathbb L}}|{\bf u}_{\boldsymbol J}|^2},
\end{equation}
where $||\cdot||$ and $||\cdot||_2$ denote the discrete $L^2$-norm and $2$-norm for a given vector, respectively. 

Then, we establish the following lemma, which significantly contributes to the theoretical analysis of the convergence of the CN-4FCD scheme (\ref{4thscheme_operator}). 
\begin{lemma}\label{Fractional Norm Equivalence}
If $1<\alpha_i<2\,(i=1 \wedge d)$, for any grid function ${\bf u}\in U_h$, there exists a positive constant $C$ independent of $h_i$ such that
\begin{equation}
 C K_{\min} \Delta t||{\bf u}||^2\leq  \langle T{\bf u},{\bf u} \rangle,
\end{equation}
where 
\begin{equation}\label{definitionforC}
C= \Big(1-\frac{1}{\pi}\Big)\Big(\frac{2}{\pi}\Big)^{\alpha_{\max}}\sum_{i=1}^{d}\frac{1}{2(\hat{a}_i-\check{a}_i)^{\alpha_i}},
\end{equation}
with 
$$
K_{\min}=\min_{i=1\wedge d}\{K_i\}, \, \alpha_{\max}=\max_{i=1\wedge d}\{\alpha_i\}.
$$
\end{lemma}
\begin{proof}
Let $(\lambda,{\bf w})$ be an eigenpair of $S_{n_{i}}^{(\alpha_{i})}$ with $\|{\bf w}\|_{2}=\sqrt{\sum_{k=0}^{n_{i}-1}|w_{k}|^{2}}=1$. 
Denote $$\psi(\theta)=\sum_{k=0}^{n_{i}-1}w_{k}e^{{\bf i} k\theta}.$$ Then, we have
\begin{align}\label{rel1}
\dfrac{1}{2\pi}\int_{-\pi}^{\pi}|\psi(\theta)|^{2}d\theta=\sum_{k=0}^{n_{i}-1}|w_{k}|^{2}=1,
\end{align}
and, by Cauchy-Schwarz inequality,
\begin{align}\label{rel2}
|\psi(\theta)|^{2}\leq\left(\sum_{k=0}^{n_{i}-1}|w_{k}|^{2}\right)\left(\sum_{k=0}^{n_{i}-1}|e^{{\bf i} k\theta}|^{2}\right)=n_{i}.
\end{align}
Besides, by $\sin(\theta/2)\geq \theta/\pi\geq 0,~\theta\in[0,\pi]$, we have
\begin{align}\label{rel3}
S^{(\alpha_{i})}(e^{\bf i\theta})=\left(1+\dfrac{\alpha_{i}}{6}\sin^{2}\dfrac{\theta}{2}\right)\left(4\sin^{2}\dfrac{\theta}{2}\right)^{\frac{\alpha_{i}}{2}}\geq\left(\dfrac{2}{\pi}\right)^{\alpha_{i}}|\theta|^{\alpha_{i}},
~\theta\in[-\pi,\pi].
\end{align}
Hence, we have
\begin{align*}
\lambda&={\bf w}^{*}S_{n_{i}}^{(\alpha_{i})}{\bf w}\\
&=\dfrac{1}{2\pi}\int_{[-\pi,\pi]}S^{(\alpha_{i})}(e^{\bf i\theta})|\psi(\theta)|^{2}d\theta\\
&\geq\dfrac{1}{2\pi}\int_{[-\pi,\pi]\setminus[-\frac{1}{n_{i}},\frac{1}{n_{i}}]}S^{(\alpha_{i})}(e^{\bf i\theta})|\psi(\theta)|^{2}d\theta\\
&\geq\Big(\frac{2}{\pi}\Big)^{\alpha_i}\dfrac{1}{2\pi}\int_{[-\pi,\pi]\setminus[-\frac{1}{n_{i}},\frac{1}{n_{i}}]}|\theta|^{\alpha_i}|\psi(\theta)|^{2}d\theta\qquad\Big(\textrm{by}~\big(\ref{rel3})\Big)
\\
&\geq\Big(\frac{2}{\pi n_i}\Big)^{\alpha_i}\left(1-\dfrac{1}{2\pi}\int_{[-\frac{1}{n_{i}},\frac{1}{n_{i}}]}|\psi(\theta)|^{2}d\theta\right)\qquad\Big(\textrm{by}~(\ref{rel1})\Big)\\
&\geq\Big(\frac{2}{\pi n_i}\Big)^{\alpha_i}\left(1-\dfrac{n_{i}}{2\pi}\int_{[-\frac{1}{n_{i}},\frac{1}{n_{i}}]}d\theta\right)\qquad\Big(\textrm{by}~(\ref{rel2})\Big)\\
&\geq\Big[\dfrac{2}{\pi(n_{i}+1)}\Big]^{\alpha_{i}}\left(1-\dfrac{1}{\pi}\right)>0.
\end{align*}
In particular, 
\begin{equation}\label{lambdaS_n}
\lambda_{\min}(S_{n_{i}}^{(\alpha_{i})})\geq\Big[\dfrac{2}{\pi(n_{i}+1)}\Big]^{\alpha_{i}}\left(1-\dfrac{1}{\pi}\right).
\end{equation}
Now, for any ${\bf u} \in U_h$, we have
\begin{align}\label{quadratic_formfor_T}
{\bf u}^{T}T{\bf u}&=
\sum\limits_{i=1}^{d}\eta_{i}{\bf u}^{T}I_{n_{i}^{-}}\otimes S_{n_{i}}^{(\alpha_{i})}\otimes I_{n_{i}^{+}}{\bf u}\nonumber\\
&\geq\sum\limits_{i=1}^{d}\eta_{i}\lambda_{\min}(S_{n_{i}}^{(\alpha_{i})}){\bf u}^{T}{\bf u}\nonumber\\
&\geq\Delta t\Big(1-\frac{1}{\pi}\Big)\sum\limits_{i=1}^{d}\dfrac{K_{i}(2/\pi)^{\alpha_i}}{2h_{i}^{\alpha_{i}}(n_{i}+1)^{\alpha_{i}}}\|{\bf u}\|_{2}^{2}\qquad\Big(\textrm{by}~(\ref{lambdaS_n})\Big)\nonumber\\
&\geq K_{\min}\Delta t\Big(1-\frac{1}{\pi}\Big)\Big(\frac{2}{\pi}\Big)^{\alpha_{\max}}\sum\limits_{i=1}^{d}\dfrac{1}{2h_{i}^{\alpha_{i}}(n_{i}+1)^{\alpha_{i}}}\|{\bf u}\|_{2}^{2}\nonumber\\
&=K_{\min}\Delta t\Big(1-\frac{1}{\pi}\Big)\Big(\frac{2}{\pi}\Big)^{\alpha_{\max}}\sum\limits_{i=1}^{d}\dfrac{1}{2(\hat{a}_{i}-\check{a}_{i})^{\alpha_{i}}}\|{\bf u}\|_{2}^{2},
\end{align}
in which 
$$
K_{\min}=\min_{i=1\wedge d}\{K_i\}\,~~{\rm and}~~ \,\alpha_{\max}=\max_{i=1\wedge d}\{\alpha_i\}.
$$
Multiplying $\prod_{i=1}^d h_i$ on both sides of (\ref{quadratic_formfor_T}), 
we obtain
\begin{equation}
\prod_{i=1}^d h_i{\bf u}^{T}T{\bf u}:=\langle T{\bf u},{\bf u} \rangle \geq  C K_{\min}\Delta t ||{\bf u}||^2,   
\end{equation}
with $C$ being defined by (\ref{definitionforC}), which completes the proof.
\end{proof}

\begin{theorem}\label{convergenceTH}
{\rm (Convergence)} Let ${\bf U}^{m+1}$ be the exact solution of (\ref{RSFDEs}) and ${\bf u}^{m+1}$ be the solution of the finite difference scheme (\ref{4thscheme_operator}) at time level $t_{m+1}$. Then, for $m=0\, \wedge\,M-1$,
we have
\begin{equation}
    ||{\bf u}^{m+1}-{\bf U}^{m+1}||\leq C_1(\Delta t^2+h_1^4+h_2^4+\cdots+h_d^4),
\end{equation}
where $C_1$ denotes a positive number independent of $h_i$ and $\Delta t$.
\end{theorem}
\begin{proof}
Denote ${\bf e}^{m+1}={\bf U}^{m+1}-{\bf u}^{m+1}$ with ${\bf e}^0={\bf 0}$, then subtracting (\ref{4thscheme_operator}) from (\ref{numerical_scheme}), we have the following matrix form for the error equation
\begin{equation}\label{errorequation}
    D^{m+\frac{1}{2}}({\bf e}^{m+1}-{\bf e}^m)+T({\bf e}^{m+1}+{\bf e}^m)=\Delta t {\boldsymbol{\rho} }^{m},
\end{equation}
where
$||{\boldsymbol{\rho}}^m||\leq c (\Delta t^2+h_1^4+h_2^4+\cdots+h_d^4)$ with  $c$ being a positive constant independent of $\Delta t$ and $h_i$.

Taking the inner product of both sides of (\ref{errorequation}) 
with ${\bf e}^{m+1}-{\bf e}^{m}$ yields 
\begin{equation}\label{erroreq2}
    \langle D^{m+\frac{1}{2}}({\bf e}^{m+1}-{\bf e}^{m}),{\bf e}^{m+1}-{\bf e}^{m}\rangle + \langle T({\bf e}^{m+1}+{\bf e}^{m}),{\bf e}^{m+1}-{\bf e}^{m}\rangle=\langle \Delta t \boldsymbol{\rho}^{m},{\bf e}^{m+1}-{\bf e}^{m}\rangle.
\end{equation}
Next, for (\ref{erroreq2}), each term will be estimated. First,
\begin{eqnarray}\label{estimation1}
   \langle D^{m+\frac{1}{2}}({\bf e}^{m+1}-{\bf e}^{m}),{\bf e}^{m+1}-{\bf e}^{m}\rangle &=& \prod_{i=1}^d h_i\sum_{{\boldsymbol J}\in {\mathbb L}}r_{{\boldsymbol J}}^{m+\frac{1}{2}}
   (e_{{\boldsymbol J}}^{m+1}-e_{{\boldsymbol J}}^{m})^2 \nonumber\\
   &\geq&\check{r}||{\bf e}^{m+1}-{\bf e}^m||^2.
\end{eqnarray}
Second, by straightforward calculation, we obtain
\begin{eqnarray}\label{estimation2}
 \langle T({\bf e}^{m+1}+{\bf e}^{m}),{\bf e}^{m+1}-{\bf e}^{m}\rangle=\langle T{\bf e}^{m+1},{\bf e}^{m+1}\rangle-\langle T{\bf e}^{m},{\bf e}^{m}\rangle.
\end{eqnarray}
Last, 
\begin{eqnarray}\label{estimation3}
  \langle \Delta t {\boldsymbol{\rho}}^{m},{\bf e}^{m+1}-{\bf e}^{m}\rangle&\leq& || {\Delta t\boldsymbol{\rho}}^m||\cdot||{\bf e}^{m+1}-{\bf e}^{m}||\nonumber\\
  &=&2\sqrt{\check{r}\cdot\frac{1}{4\check{r}}}||{\Delta t\boldsymbol{\rho}}^m||\cdot||{\bf e}^{m+1}-{\bf e}^{m}||\nonumber\\
  &\leq& \check{r}||{\bf e}^{m+1}-{\bf e}^{m}||^2+\frac{1}{4\check{r}}||{\Delta t\boldsymbol{\rho}}^{m}||^2.
\end{eqnarray}
Based on Equations (\ref{erroreq2})-(\ref{estimation3}), we obtain
\begin{eqnarray}\label{estimation4}
   \langle T{\bf e}^{m+1},{\bf e}^{m+1}\rangle-\langle T{\bf e}^{m},{\bf e}^{m}\rangle \leq \frac{1}{4\check{r}}||{\Delta t\boldsymbol{\rho}}^{m}||^2.
\end{eqnarray}
For (\ref{estimation4}), summing up for $m$ from $0$ to $k$ yields
\begin{eqnarray}\label{estimation5}
   \langle T {\bf e}^{k+1},{\bf e}^{k+1} \rangle&\leq&\langle T {\bf e}^{0},{\bf e}^{0} \rangle + \frac{1}{4\check{r}}\sum_{m=0}^k ||{\Delta t\boldsymbol{\rho}}^m||^2,\nonumber\\
   &=&\frac{1}{4\check{r}}\sum_{m=0}^k ||{\Delta t\boldsymbol{\rho}}^m||^2,\quad k\in 0 \wedge M-1.
\end{eqnarray}
According to Lemma \ref{Fractional Norm Equivalence}, (\ref{estimation4}) is discredited into
\begin{eqnarray}
    ||{\bf e}^{k+1}||^2 &\leq&
    \frac{1}{C K_{\min} \Delta t}  \langle T{\bf e}^{k+1},{\bf e}^{k+1} \rangle\nonumber\\
    &\leq&\frac{1}{4\check{r}CK_{\min}\Delta t}\sum_{m=0}^k ||{\Delta t\boldsymbol{\rho}}^m||^2\nonumber\\
    &\leq& \frac{c^2}{4\check{r}CK_{\min}\Delta t}(k+1)\Delta t^2(\Delta t^2+h_1^4+h_2^4+\cdots+ h_d^4)^2\nonumber\\
    &\leq&\frac{c^2T}{4\check{r}CK_{\min}}(\Delta t^2+h_1^4+h_2^4+\cdots+ h_d^4)^2.
\end{eqnarray}
That is 
\begin{eqnarray}
    ||{\bf e}^{k+1}||\leq C_1(\Delta t^2+h_1^4+h_2^4+\cdots+h_d^4),
\end{eqnarray}
in which $C_1=c\sqrt{\frac{T}{4\check{r}CK_{\min}}}$ with $C$ being defined by (\ref{definitionforC}). Replacing $k$ with $m$, then Theorem \ref{convergenceTH} is proved.
\end{proof}
\begin{remark}
Theorems \ref{stability theorem} and \ref{convergenceTH} shows that the CN-4FCD scheme is stable and can achieve second-order accuracy in time and fourth-order accuracy in space under the discrete $L^2$-norm without any restrictions on the spatial step size. In contrast, the stability and convergence of the quasi-compact scheme requires that the spatial step is small. 
For more details, see Theorems 3.1 and 3.2 in \cite{zhang2024two}. 
\end{remark}

\section{Sine Transform
Based Preconditioner and Its Spectral Analysis}\label{Sec-3}
Based on the diagonal-plus-Toeplitz structure of the coefficient matrix in (\ref{matrix-form}),
our sine transform based preconditioner is defined as
\begin{equation}\label{ptau}
P_{\tau}^{m+1}=\bar r^{m+\frac{1}{2}} + \tau(T).
\end{equation}
Here $\bar r^{m+\frac{1}{2}}$ is a kind of average of $r(\boldsymbol{x},t_{m+\frac{1}{2}})$ and can be chosen as (but not limited to) $\sqrt{\hat{r}^{m+\frac{1}{2}}\check{r}^{m+\frac{1}{2}}}$ or $\frac{\hat{r}^{m+\frac{1}{2}}+\check{r}^{m+\frac{1}{2}}}{2}$
with
\begin{equation}
\hat{r}^{m+\frac{1}{2}}=\max r\left(\boldsymbol{x}, t_{m+ \frac{1}{2}}\right), \quad \check{r}^{m+\frac{1}{2}}=\min r\left(\boldsymbol{x}, t_{m+\frac{1}{2}}\right),
\end{equation}
and 
\begin{equation}\label{ptau1}
\tau(T)=\sum_{i=1}^d \eta_i I_{n_i^{-}} \otimes P_{\alpha_i} \otimes I_{n_i^{+}}.
\end{equation}
It is remarked that $P_{\alpha_i}$ is the sine transform based preconditioner proposed in \cite{qu2025novel}, which is defined as 
$P_{\alpha_i}:=Q_{n_i} \tau(\hat S_{n_i}^{\left(\alpha_i\right)})$
with
\begin{equation}\label{}
Q_{n_i}=\left(\begin{array}{ccccc}
1 &  & & & \\
 & 1 &  & & \\
&  & \ddots &  & \\
& &  & 1 &  \\
& & &  & 1
\end{array}\right)+\frac{\alpha_i}{24}\left(\begin{array}{ccccc}
2 & -1 & & & \\
-1 & 2 & -1 & & \\
& \ddots & \ddots & \ddots & \\
& & -1 & 2 & -1 \\
& & & -1 & 2
\end{array}\right) \in \mathbb{R}^{n_i\times n_i}
\end{equation}
and $\hat S_{n_i}^{\left(\alpha_i\right)}$ is the spatial discretization matrix from the fractional centered difference approximation for the Riesz fractional derivatives \cite{CD2,zhang2022fast}.

\begin{proposition}\cite{lin2018efficient}\label{main_ineq}
For positive numbers $\xi_i$ and $\zeta_i$ $(1\leq i \leq m)$, it holds that
\begin{equation*}
\min _{1 \leq i \leq m} \frac{\xi_i}{\zeta_i} \leq\left(\sum_{i=1}^m \zeta_i\right)^{-1}\left(\sum_{i=1}^m \xi_i\right) \leq \max _{1 \leq i \leq m} \frac{\xi_i}{\zeta_i}.
\end{equation*}
\end{proposition}

\begin{lemma}\label{pre_SPD}
The sine transform based preconditioner $P_{\tau}$ defined in (\ref{ptau}) is SPD and thus $P_{\tau}$ is invertible.
\end{lemma}
\begin{proof}
It has been shown in \cite{qu2025novel} that $P_{\alpha_i}$ is SPD. Combining (\ref{ptau}) with (\ref{ptau1}) and noticing that $\eta_i>0$, the result can be obtained directly using the properties of Kronecker product.
\end{proof}

\begin{lemma}\label{1D_bound}\cite{qu2025novel}
Let $\lambda(P_{\alpha_i}^{-1}S_{n_i}^{\left(\alpha_i\right)})$ be an eigenvalue of the matrix $P_{\alpha_i}^{-1}S_{n_i}^{\left(\alpha_i\right)}$. Then,
\begin{equation}\nonumber
\frac{3}{8}\leq\lambda(P_{\alpha_i}^{-1}S_{n_i}^{\left(\alpha_i\right)})\le2.
\end{equation}
\end{lemma}

\begin{theorem}\label{eigenprematrix}
Let $\lambda(P_\tau^{-1}(D^{m+\frac{1}{2}}+T))$ be an eigenvalue of the matrix $P_\tau^{-1}(D^{m+\frac{1}{2}}+T)$. Then,
\begin{equation}\nonumber
\min \left\{\frac{\check{r}^{m+\frac{1}{2}}}{\hat{r}^{m+\frac{1}{2}}}, \frac{3}{8}\right\} \leq \lambda(P_\tau^{-1}(D^{m+\frac{1}{2}}+T)) \leq \max \left\{\frac{\hat{r}^{m+\frac{1}{2}}}{\check{r}^{m+\frac{1}{2}}}, 2\right\}.
\end{equation}
\end{theorem}
\begin{proof}
For $i= 1,2,\ldots d$, by the Rayleigh quotient theorem and Lemma \ref{1D_bound}, it holds that
\begin{equation}\label{matrix_ineq}
\frac{3}{8}<\lambda_{\min}\big(I_{n_i^{-}} \otimes P_{\alpha_i}^{-1} S_{n_i}^{\left(\alpha_i\right)} \otimes I_{n_i^{+}}\big)\leq\frac{v^{T}(I_{n_i^{-}} \otimes S_{n_i}^{\left(\alpha_i\right)} \otimes I_{n_i^{+}})v}{v^{T} (I_{n_i^{-}} \otimes P_{\alpha_i} \otimes I_{n_i^{+}}) v}\leq \lambda_{\max}\big(I_{n_i^{-}} \otimes P_{\alpha_i}^{-1} S_{n_i}^{\left(\alpha_i\right)} \otimes I_{n_i^{+}}\big)<2.
\end{equation}
On the other hand, from the construction of $P_{\tau}^{m+1}$, it is obvious that 
\begin{equation}\label{coef_ineq}
\frac{\check{r}^{m+\frac{1}{2}}}{\hat{r}^{m+\frac{1}{2}}} \leq\frac{v^{T} D^{m+\frac{1}{2}} v}{v^{T} \bar r^{m+\frac{1}{2}} v}\leq \frac{\hat{r}^{m+\frac{1}{2}}}{\check{r}^{m+\frac{1}{2}}}.
\end{equation}
Combining Equations (\ref{matrix_ineq}), (\ref{coef_ineq}) with Proposition \ref{main_ineq}, we have
\begin{equation}\nonumber
\min \left\{\frac{\check{r}^{m+\frac{1}{2}}}{\hat{r}^{m+\frac{1}{2}}}, \frac{3}{8}\right\} \leq \frac{v^{T}(D^{m+\frac{1}{2}}+\sum_{i=1}^d \eta_i I_{n_i^{-}} \otimes S_{n_i}^{\left(\alpha_i\right)} \otimes I_{n_i^{+}}) v}{v^{T} (\bar r^{m+\frac{1}{2}}+\sum_{i=1}^d \eta_i I_{n_i^{-}} \otimes P_{\alpha_i} \otimes I_{n_i^{+}})) v} \leq \max \left\{\frac{\hat{r}^{m+\frac{1}{2}}}{\check{r}^{m+\frac{1}{2}}}, 2\right\},
\end{equation}
i.e.,
\begin{equation}\nonumber
\min \left\{\frac{\check{r}^{m+\frac{1}{2}}}{\hat{r}^{m+\frac{1}{2}}}, \frac{3}{8}\right\} \leq \frac{v^{T} (D^{m+\frac{1}{2}}+T) v}{v^{T} P_{\tau}^{m+1} v} \leq \max \left\{\frac{\hat{r}^{m+\frac{1}{2}}}{\check{r}^{m+\frac{1}{2}}}, 2\right\}.
\end{equation}
Therefore, we have
\begin{equation}\nonumber
\lambda_{\min}\Big(P_\tau^{-1}(D^{m+\frac{1}{2}}+T)\Big)=\min\frac{v^{T} (D^{m+\frac{1}{2}}+T) v}{v^{T} P_{\tau}^{m+1} v} \geq \min \left\{\frac{\check{r}^{m+\frac{1}{2}}}{\hat{r}^{m+\frac{1}{2}}}, \frac{3}{8}\right\}
\end{equation}
and
\begin{equation}\nonumber
\lambda_{\max}\Big(P_\tau^{-1}(D^{m+\frac{1}{2}}+T)\Big)=\max\frac{v^{T} (D^{m+\frac{1}{2}}+T) v}{v^{T} P_{\tau}^{m+1} v} \leq \max \left\{\frac{\hat{r}^{m+\frac{1}{2}}}{\check{r}^{m+\frac{1}{2}}}, 2\right\},
\end{equation}
which completes the proof.
\end{proof}

\section{Numerical Experiments}\label{sec:num}
In this section, we choose two examples to show the accuracy of the proposed scheme and the efficiency of the preconditioner. All numerical experiments are performed using MATLAB R2019a on a Lenovo laptop with 16GB RAM, AMD Ryzen 5 4600U with Radeon Graphics $@$ 2.10 GHz.

Since the coefficient matrix in (\ref{matrix-form}) is SPD, we choose the CG method to solve (\ref{matrix-form}), which is defined as `CG'. In addition, the proposed sine transform based preconditioner defined in (\ref{ptau}) can be used to accelerate the convergence of the CG method. We denote it as `P$_\tau$-CG'. For comparisons, two commonly used circulant-type preconditioners (Strang's circulant preconditioner and T. Chan's circulant preconditioner) are also adopted for the CG method, which are denoted as `P$_S$-CG' and `P$_T$-CG', respectively.
Furthermore, three different kinds of induced preconditioners by splitting methods are also considered. The first two are the Strang's circulant preconditioner and the sine transform based preconditioner induced from the splitting method proposed in \cite{bai2017diagonal,bai2020fast}; another one is the sine transform based preconditioner induced from the splitting method proposed in \cite{tang2024new}. Since the induced preconditioners are not symmetric, we solve the system (\ref{matrix-form}) using GMRES with these preconditioners and denote them as `P$_S$-GMRES', `P$_{\tau_1}$-GMRES' and `P$_{\tau_2}$-GMRES', respectively. Last, the fourth-order scheme and the corresponding $\tau$ preconditioner proposed in \cite{she2024tau} are also adopted as a comparison. Although the scheme in \cite{she2024tau} is also fourth-order accurate in space (achieved via compact operators \cite{hao2015fourth}), it generates non-symmetric linear systems \cite{she2024tau}. We therefore solve the preconditioned systems with GMRES, denoted as `P$_{\tau_c}$-GMRES'.
 
In the implementations of CG and GMRES methods, we set $h_1=h_2=\cdots=h_d=h$ and adopt the MATLAB built-in functions $\mathbf{pcg}$ and $\mathbf{gmres}$ with $\mathit{maxit}=\prod\limits_{i=1}^d n_i$ and $\mathit{restart}=30$. The initial guess for all methods at each time step is chosen as the zero vector, and the stopping criterion is set as
$$
\frac{\|\mathbf{r}_k\|_2}{\|\mathbf{r}_0\|_2}\leq 10^{-9},
$$
where $\mathbf{r}_k$ denotes the residual vector at the $k$-th iteration. The parameters in the induced preconditioners are all set as 0.1, while $\bar r^{m+\frac{1}{2}}$ in the preconditioners for `P$_T$-CG', `P$_S$-CG', `P$_\tau$-CG' and `P$_{\tau_c}$-GMRES' are chosen as $\bar r^{m+\frac{1}{2}}=\frac{\hat{r}^{m+\frac{1}{2}}+\check{r}^{m+\frac{1}{2}}}{2}$. Furthermore, to reduce operation costs, all matrix-vector multiplications are fast evaluated via the MATLAB built-in functions {\bf fft}, {\bf ifft} and {\bf dst}.

In all tables, `${\rm E}(h,\Delta t)$' represents the error of the numerical solution and the numerical solution under the discrete $L_2$-norm. The convergence orders in time and space are measured as follows
\begin{equation*}
R_{\Delta t}=\log _{2}\left(\frac{\operatorname{E}\left(h, \Delta t\right)}{\operatorname{E}\left(h, \frac{\Delta t}{2}\right)}\right), \,R_{h}=\log _{2}\left(\frac{\operatorname{E}\left(h, \Delta t\right)}{\operatorname{E}\left(\frac{h}{2}, \Delta t\right)}\right).
\end{equation*}
`CPU(s)' represents the CPU time in seconds for solving the related system; `Iter' stands for the iteration numbers of different methods; When CPU time is large than 10000 seconds, we stop the iteration and denote the corresponding iteration number as `$\dagger$'. 

\begin{example}\label{example1}\cite{she2024tau}
Consider problem (\ref{RSFDEs}) with $d=2$, $\Omega=(0,1)^2$, $T=1$, $K_1=K_2=100$ and the exact solution $u(x_1,x_2,t)=10^4e^{-t}x_1^4(1-x_1)^4x_2^4(1-x_2)^4$. The variable coefficient function is set as $r(x_1,x_2,t)=\frac{x_1^2+x_2^2+e^{-t}}{100}$.
\end{example}


\begin{figure}[htbp]
	\centering
	\begin{minipage}{0.49\linewidth}   
		\centering
		\includegraphics[width=1.1\linewidth]{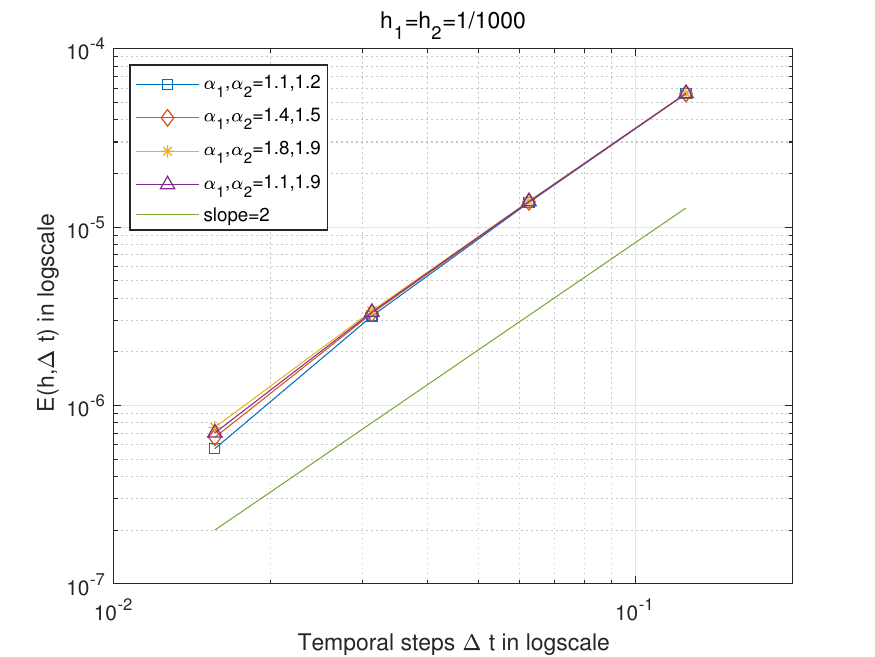}
	\end{minipage}
	\begin{minipage}{0.49\linewidth}
		\centering
		\includegraphics[width=1.1\linewidth]{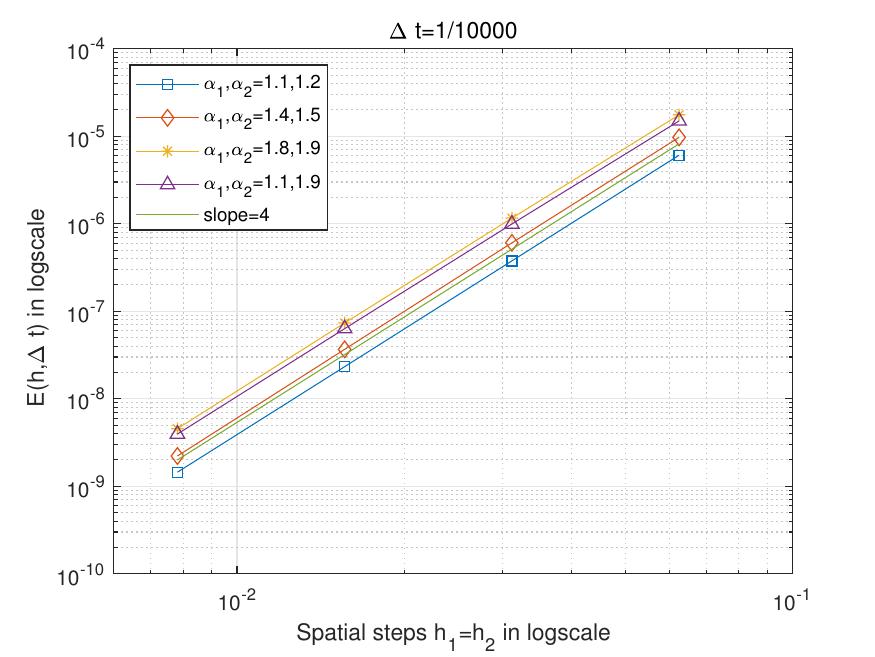}
	\end{minipage}
	\caption{Temporal convergence rate (left) and spatial convergence rate (right) for different fractional orders in Example \ref{example1} at $T=1$.}
	\label{fig1}
\end{figure}

For different orders $\alpha_1$ and $\alpha_2$ ranging from $1$ to $2$, the errors and the corresponding temporal and spatial convergence orders of the proposed scheme are shown in Figure \ref{fig1}. It can be observed that the proposed scheme is stable and achieves a second-order convergence rate in time and a fourth-order convergence rate in space, which aligns well with the theoretical results given in Theorems \ref{stability theorem} and \ref{convergenceTH}. In addition, the CPU time and iteration numbers obtained using the different methods mentioned at the beginning of this section are given in Tables \ref{table3} and \ref{table4}. From the tables, we have the following observations:

\begin{itemize}
\item[(i)] The iteration numbers and CPU time of all preconditioned methods that solve the system (\ref{matrix-form}) can be greatly reduced.

\item[(ii)] The iteration numbers obtained by the preconditioned CG/GMRES methods with the sine transform based preconditioners remain almost constant as the matrix size varies, whereas those obtained with the circulant-type preconditioners show a noticeable increase. This leads to a significant rise in the corresponding computation time as the matrix size increases.

\item[(iii)] Although all the preconditioned solvers with the sine transform based preconditioners show the size-independent iteration numbers, the CPU time of our proposed `P$_\tau$-CG' method requires the least CPU time. This is because, compared to the proposed preconditioner $P_{\tau}^{m+1}$ defined in \eqref{ptau}, other sine transform based preconditioners have more complicated structures and are not as easy to implement as $P_{\tau}^{m+1}$, especially the preconditioners used in `P$_{\tau_c}$-GMRES', see \cite{bai2020fast,bai2017diagonal,tang2024new,she2024tau} for more details. Besides, for `P$_{\tau_c}$-GMRES',  in addition to the complicated structure of the preconditioner, the process in generating the right-hand side of the linear system is also complicated due to the introduction of the compact operator; see, for example, \cite{she2024tau}. Although all the preconditioned solvers with sine transform based preconditioners show the size-free iteration numbers, the simple structure and easy-to-implement features ensure that our method requires minimal computation time, which reflects one of the advantages of our method.
\end{itemize}

To further illustrate the efficiency of our proposed preconditioner, we plot the eigenvalue distributions of the preconditioned matrices in Figure \ref{fig2}. Unlike the CG method, the convergence of GMRES cannot be determined from eigenvalues alone \cite{greenbaum1997iterative}. Therefore, only the spectra of the preconditioned matrices related to the CG methods are plotted. As shown in Figure \ref{fig2}, the eigenvalues of the preconditioned matrices become clustered compared to those of the original coefficient matrix. In addition, the spectrum of the preconditioned matrix of P$_\tau$-CG are bounded, more clustered and farther away from zero compared to those with other two circulant type preconditioners, which illustrates the superior performance of P$_\tau$-CG exibited in Tables \ref{table3} and \ref{table4}, see, for example \cite{chan2007introduction}.

\begin{table}[!tbp]
\centering
\tabcolsep=5pt
\renewcommand{\arraystretch}{0.5}
\caption{Comparison of CPU time and iteration numbers for different $M$, $N$ and orders $(\alpha_1,\alpha_2)=(1.1,1.2)$ and $(\alpha_1,\alpha_2)=(1.4,1.5)$ in Example \ref{example1} at $T=1$.}
\label{table3}
\begin{tabular}{ccccccc}
\hline
\multirow{2}{*}{Methods}& \multirow{2}{*}{$M$}& \multirow{2}{*}{$N$}&\multicolumn{2}{c}{$(\alpha_1,\alpha_2)=(1.1,1.2)$} &\multicolumn{2}{c}{$(\alpha_1,\alpha_2)=(1.4,1.5)$}
   \\ \cmidrule(r){4-5} \cmidrule(r){6-7}
 &&  &CPU(s)  &Iter &CPU(s)  &Iter  \\
\hline 
\multirow{4}{*}{CG}        
         &$2^{3}$&$2^{8}$  &132.46    &243.00      &279.35       &510.94          \\
         &$2^{4}$&$2^{9}$  &1570.70   &370.00      &3706.82      &884.56       \\
         &$2^{5}$&$2^{10}$ &$>10000$  &$\dagger$   &$>10000$     &$\dagger$         \\
\hline          
 \multirow{4}{*}{P$_\tau$-CG}        
         &$2^{3}$&$2^{8}$  &14.04     &10.00      &12.47     &9.00          \\
         &$2^{4}$&$2^{9}$  &101.37    &10.00      &100.44    &10.00       \\
         &$2^{5}$&$2^{10}$ &903.59    &11.00      &829.64    &10.00         \\         
\hline          
 \multirow{4}{*}{P$_T$-CG}        
         &$2^{3}$&$2^{8}$  &21.91     &31.00      &36.17     &53.06          \\
         &$2^{4}$&$2^{9}$  &193.27    &36.00      &378.26    &73.00       \\
         &$2^{5}$&$2^{10}$ &2060.28   &42.47      &4633.21   &98.91         \\
\hline          
 \multirow{4}{*}{P$_S$-CG}        
         &$2^{3}$&$2^{8}$  &18.57    &25.00      &23.81     &33.00          \\
         &$2^{4}$&$2^{9}$  &156.61   &28.00      &224.36    &40.97       \\
         &$2^{5}$&$2^{10}$ &1646.42  &31.98      &2504.94   &48.77         \\
\hline          
 \multirow{4}{*}{P$_S$-GMRES}        
         &$2^{3}$&$2^{8}$  &22.19     &24.00      &31.84     &33.00          \\
         &$2^{4}$&$2^{9}$  &256.41    &28.00      &346.09    &40.00       \\
         &$2^{5}$&$2^{10}$ &2405.11   &30.00      &3614.45   &49.00         \\
\hline          
 \multirow{4}{*}{P$_{\tau_1}$-GMRES}        
         &$2^{3}$&$2^{8}$  &16.77     &10.00      &14.53     &9.00          \\
         &$2^{4}$&$2^{9}$  &136.21    &10.00      &114.51    &9.00       \\
         &$2^{5}$&$2^{10}$ &1035.46   &10.00      &1033.35   &10.00         \\
\hline          
 \multirow{4}{*}{P$_{\tau_2}$-GMRES}        
         &$2^{3}$&$2^{8}$  &16.03     &10.00      &14.65     &9.00          \\
         &$2^{4}$&$2^{9}$  &125.70    &10.00      &125.08    &9.69       \\
         &$2^{5}$&$2^{10}$ &1101.65   &10.73      &1033.42   &10.00         \\
\hline          
 \multirow{4}{*}{P$_{\tau_c}$-GMRES}        
         &$2^{3}$&$2^{8}$  &24.58     &9.00      &22.66    &8.00          \\
         &$2^{4}$&$2^{9}$  &192.21    &9.00      &173.01   &8.00       \\
         &$2^{5}$&$2^{10}$ &1544.60   &9.00      &1413.49  &8.00         \\
\hline
\end{tabular}
\end{table} 

\begin{table}[!tbp]
\centering
\tabcolsep=5pt
\renewcommand{\arraystretch}{0.5}
\caption{Comparison of CPU time and iteration numbers for different $M$, $N$ and orders $(\alpha_1,\alpha_2)=(1.8,1.9)$ and $(\alpha_1,\alpha_2)=(1.1,1.9)$ in Example \ref{example1} at $T=1$.}
\label{table4}
\begin{tabular}{ccccccc}
\hline
\multirow{2}{*}{Methods}& \multirow{2}{*}{$M$}& \multirow{2}{*}{$N$}&\multicolumn{2}{c}{$(\alpha_1,\alpha_2)=(1.8,1.9)$} &\multicolumn{2}{c}{$(\alpha_1,\alpha_2)=(1.1,1.9)$}
   \\ \cmidrule(r){4-5} \cmidrule(r){6-7}
 &&  &CPU(s)  &Iter &CPU(s)  &Iter  \\
\hline 
\multirow{4}{*}{CG}        
         &$2^{3}$&$2^{8}$  &669.03    &1258.94      &910.38    &1721.00          \\
         &$2^{4}$&$2^{9}$  &$>10000$  &$\dagger$    &$>10000$  &$\dagger$       \\
         &$2^{5}$&$2^{10}$ &$>10000$  &$\dagger$    &$>10000$  &$\dagger$         \\
\hline          
 \multirow{4}{*}{P$_\tau$-CG}        
         &$2^{3}$&$2^{8}$  &10.13     &7.00      &12.47     &9.00          \\
         &$2^{4}$&$2^{9}$  &83.32     &8.00      &100.44    &10.00       \\
         &$2^{5}$&$2^{10}$ &686.96    &8.00      &829.96    &10.00         \\         
\hline          
 \multirow{4}{*}{P$_T$-CG}        
         &$2^{3}$&$2^{8}$  &98.85     &134.06         &109.98    &154.00          \\
         &$2^{4}$&$2^{9}$  &1237.89   &225.16         &1256.76   &236.00       \\
         &$2^{5}$&$2^{10}$ &$>10000$  &$\dagger$      &$>10000$  &$\dagger$         \\
\hline          
 \multirow{4}{*}{P$_S$-CG}        
         &$2^{3}$&$2^{8}$  &32.50    &46.00      &48.29     &71.81          \\
         &$2^{4}$&$2^{9}$  &344.21   &64.31      &534.69    &104.75       \\
         &$2^{5}$&$2^{10}$ &3951.71  &82.25      &6720.23   &144.34         \\
\hline          
 \multirow{4}{*}{P$_S$-GMRES}        
         &$2^{3}$&$2^{8}$  &44.66     &48.00      &88.71     &94.00          \\
         &$2^{4}$&$2^{9}$  &517.47    &58.44      &1093.69   &124.97       \\
         &$2^{5}$&$2^{10}$ &5670.66   &76.78      &$>10000$  &$\dagger$         \\
\hline          
 \multirow{4}{*}{P$_{\tau_1}$-GMRES}        
         &$2^{3}$&$2^{8}$  &15.76    &7.00      &14.70     &9.00          \\
         &$2^{4}$&$2^{9}$  &92.63    &7.00      &117.35    &9.00       \\
         &$2^{5}$&$2^{10}$ &757.49   &7.00      &929.71    &9.00         \\
\hline          
 \multirow{4}{*}{P$_{\tau_2}$-GMRES}        
         &$2^{3}$&$2^{8}$  &11.99    &7.00      &14.70     &9.00          \\
         &$2^{4}$&$2^{9}$  &92.71    &7.00      &114.24    &9.69       \\
         &$2^{5}$&$2^{10}$ &757.49   &7.00      &927.55    &9.00         \\
\hline          
 \multirow{4}{*}{P$_{\tau_c}$-GMRES}        
         &$2^{3}$&$2^{8}$  &18.51     &6.00      &20.52    &7.00          \\
         &$2^{4}$&$2^{9}$  &140.34    &6.00      &156.85   &7.00       \\
         &$2^{5}$&$2^{10}$ &1149.98   &6.00      &1409.87  &8.00         \\
\hline
\end{tabular}
\end{table}

\begin{figure}[htbp]
	\centering
	\begin{minipage}{0.49\linewidth}   
		\centering
		\includegraphics[width=1.09\linewidth]{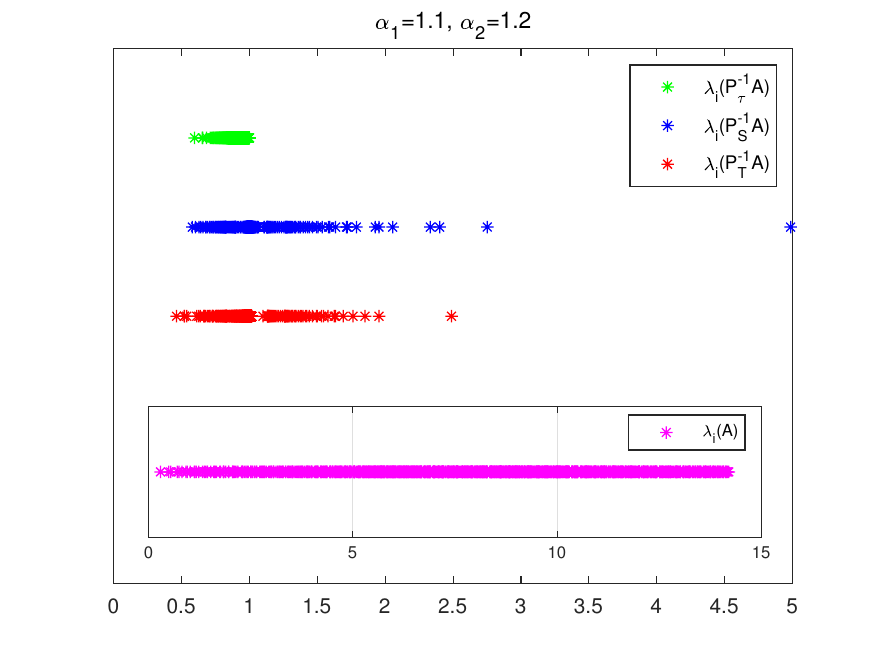}
	\end{minipage}
	\begin{minipage}{0.49\linewidth}
		\centering
		\includegraphics[width=1.09\linewidth]{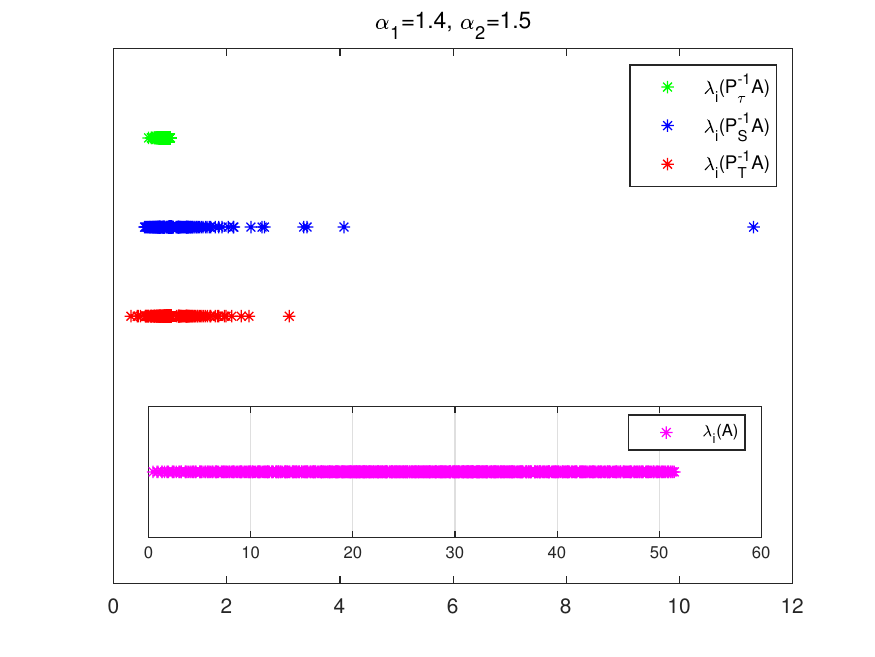}
	\end{minipage}
   \begin{minipage}{0.49\linewidth}   
		\centering
		\includegraphics[width=1.09\linewidth]{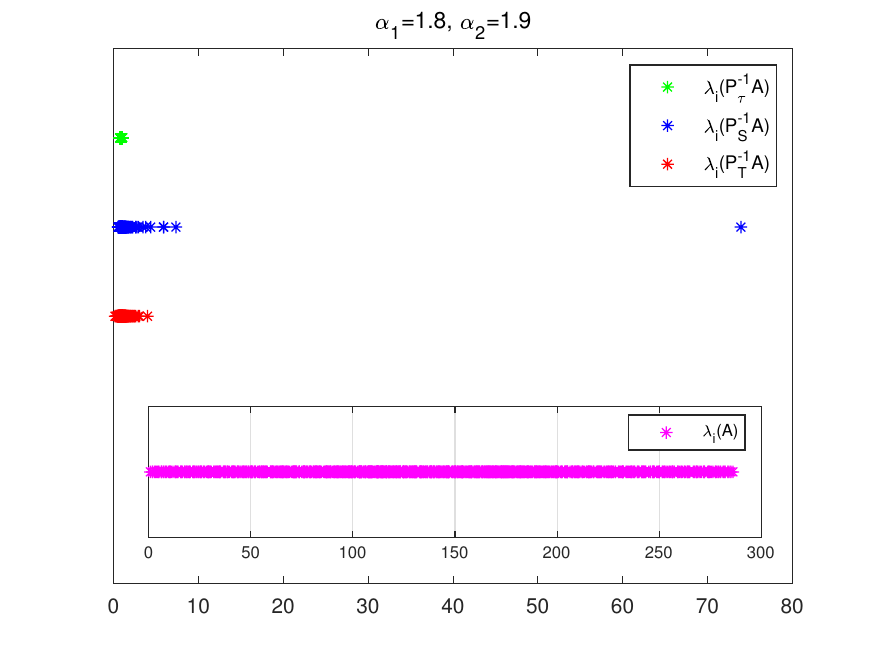}
	\end{minipage}
	\begin{minipage}{0.49\linewidth}
		\centering
		\includegraphics[width=1.09\linewidth]{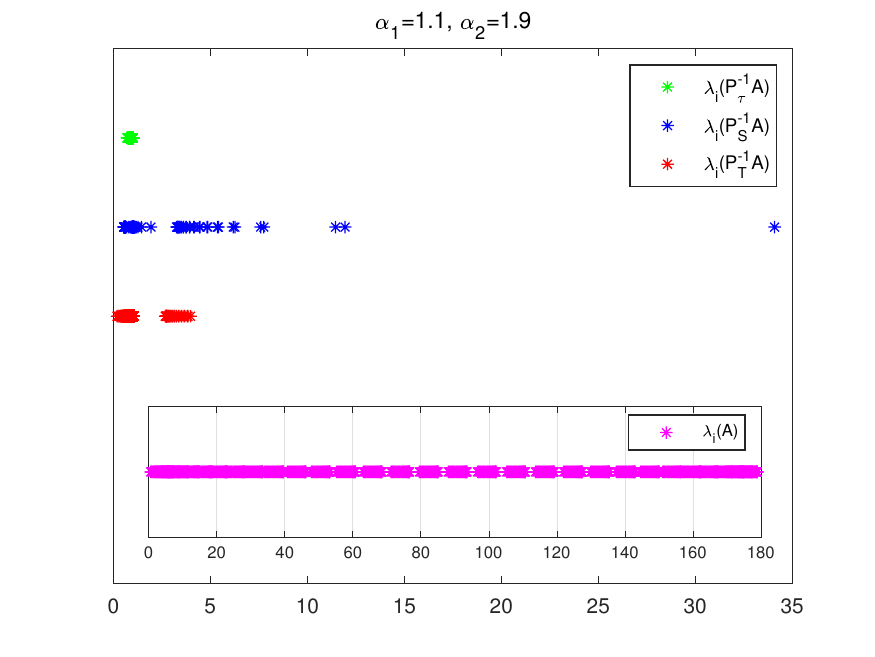}
	\end{minipage} 
	\caption{The eigenvalue distribution of the matrices related to CG and PCG methods in Example \ref{example1} at the last time step with $n_1+1=n_2+1=2^5, M=2^{10}$ for difference fractional orders.}
\label{fig2}       
\end{figure}

\begin{example}\label{example2}\cite{she2024tau}
Consider problem (\ref{RSFDEs}) with $d=3$, $\Omega=(0,1)^3$, $T=1$, $K_1=100$, $K_2=85$, $K_2=103$ and the exact solution $u(x_1,x_2,x_3,t)=10^8e^{-t}x_1^4(1-x_1)^4x_2^4(1-x_2)^4x_3^4(1-x_3)^4$. The variable coefficient function is set as $r(x_1,x_2,x_3,t)=\frac{x_1^2+x_2^2+x_3^2+e^{-t}}{100}$.
\end{example}

\begin{figure}[htbp]
	\centering
	\begin{minipage}{0.49\linewidth}   
		\centering
		\includegraphics[width=1.1\linewidth]{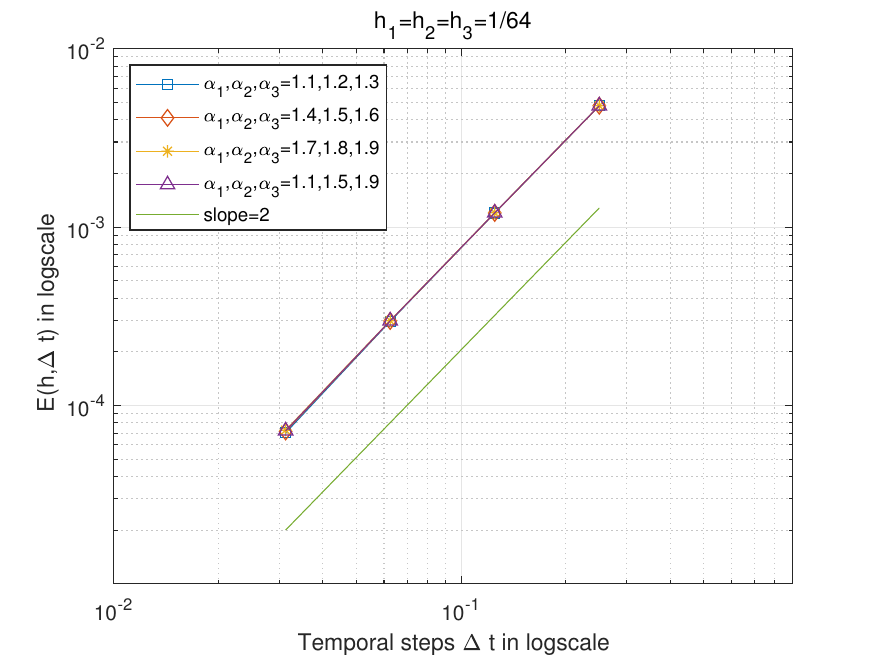}
	\end{minipage}
	\begin{minipage}{0.49\linewidth}
		\centering
		\includegraphics[width=1.1\linewidth]{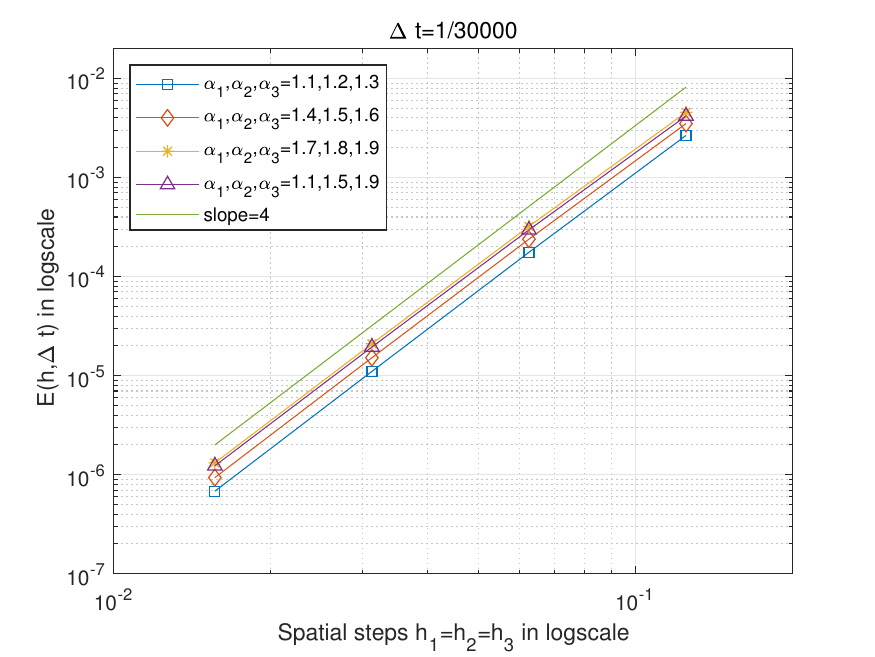}
	\end{minipage}
	\caption{Temporal convergence rate (left) and spatial convergence rate (right) for different fractional orders in Example \ref{example2} at $T=1$.}
	\label{fig3}
\end{figure}

Similar to the two-dimensional problem, for the three-dimensional problem described in the Example \ref{example2}, we first depict the errors and convergence rates for different choices of orders from 1 to 2 when temporal/spatial step size varies in Figure \ref{fig3}. Obviously, the scheme achieves a second-order convergence rate in time and a fourth-order convergence rate in space, which is consistent with our theoretical analysis. Also, the CPU times and iteration numbers of different methods with different fractional orders are listed in Tables \ref{table7} and \ref{table8}. As observed from the two tables, the iteration numbers and CPU times are significantly reduced when the preconditioners are applied. The iteration numbers by CG/GMRES methods with sine transform based preconditioners are stable, contrasting with the significantly increased number of iterations of CG/GMRES methods with the circulant-type preconditioners. Furthermore, `P$_\tau$-CG' performs the best among all methods, with a constant number of iterations and the shortest computation time.
The remarkable performance of the proposed preconditioning strategy can be intuitively illustrated by the bounded and strongly clustered eigenvalues of the preconditioned matrix, as depicted in Figure \ref{fig4}.

\begin{table}[!tbp]
\centering
\tabcolsep=5pt
\renewcommand{\arraystretch}{0.5}
\caption{Comparison of CPU time and iteration numbers for different $M$, $N$ and orders $(\alpha_1,\alpha_2,\alpha_3)=(1.1,1.3,1.5)$ and $(\alpha_1,\alpha_2,\alpha_3)=(1.3,1.5,1.7)$ in Example \ref{example2} at $T=1$.}
\label{table7}
\begin{tabular}{ccccccc}
\hline
\multirow{2}{*}{Methods}& \multirow{2}{*}{$M$}& \multirow{2}{*}{$N$}&\multicolumn{2}{c}{$(\alpha_1,\alpha_2,\alpha_3)=(1.1,1.3,1.5)$} &\multicolumn{2}{c}{$(\alpha_1,\alpha_2,\alpha_3)=(1.3,1.5,1.7)$}
   \\ \cmidrule(r){4-5} \cmidrule(r){6-7}
 &&  &CPU(s)  &Iter &CPU(s)  &Iter  \\
\hline 
\multirow{4}{*}{CG}        
         &$2^{4}$&$2^{6}$  &86.63     &120.88      &118.78    &166.19          \\
         &$2^{5}$&$2^{7}$  &2401.51   &208.97      &3555.00   &309.97       \\
         &$2^{6}$&$2^{8}$  &$>10000$  &$\dagger$   &$>10000$  &$\dagger$         \\
\hline          
 \multirow{4}{*}{P$_\tau$-CG}        
         &$2^{4}$&$2^{6}$ &15.24     &8.00      &13.80     &8.00          \\
         &$2^{5}$&$2^{7}$ &251.74    &9.00      &228.63    &8.00       \\
         &$2^{6}$&$2^{8}$ &3971.93   &9.00      &4066.01   &9.00         \\         
\hline          
 \multirow{4}{*}{P$_T$-CG}        
         &$2^{4}$&$2^{6}$ &25.95     &30.00      &31.41     &37.00          \\
         &$2^{5}$&$2^{7}$ &613.14    &39.00      &818.11    &54.00       \\
         &$2^{6}$&$2^{8}$ &$>10000$  &$\dagger$  &$>10000$  &$\dagger$        \\
\hline          
 \multirow{4}{*}{P$_S$-CG}        
         &$2^{4}$&$2^{6}$ &22.17    &25.00      &25.18     &29.06          \\
         &$2^{5}$&$2^{7}$ &458.66   &31.00      &593.46    &39.00       \\
         &$2^{6}$&$2^{8}$ &8332.15  &38.25      &$>10000$  &$\dagger$         \\
\hline          
 \multirow{4}{*}{P$_S$-GMRES}        
        &$2^{4}$&$2^{6}$ &29.33     &28.00       &35.84     &34.00          \\
        &$2^{5}$&$2^{7}$ &721.48    &34.00       &900.20    &42.59       \\
        &$2^{6}$&$2^{8}$ &$>10000$  &$\dagger$   &$>10000$  &$\dagger$         \\
\hline          
 \multirow{4}{*}{P$_{\tau_1}$-GMRES}        
        &$2^{4}$&$2^{6}$ &19.19     &9.00       &18.55     &9.00          \\
        &$2^{5}$&$2^{7}$ &328.92    &10.00      &300.54    &9.00       \\
        &$2^{6}$&$2^{8}$ &5371.95   &10.00      &4893.83   &9.00         \\
\hline          
 \multirow{4}{*}{P$_{\tau_2}$-GMRES}        
         &$2^{4}$&$2^{6}$ &19.23     &9.00       &19.13     &9.00          \\
         &$2^{5}$&$2^{7}$ &328.90    &10.00      &300.73    &9.00       \\
         &$2^{6}$&$2^{8}$ &5332.72   &10.00      &4868.79   &9.00         \\
\hline          
 \multirow{4}{*}{P$_{\tau_c}$-GMRES}        
         &$2^{4}$&$2^{6}$ &36.70     &8.00      &33.55    &7.00          \\
         &$2^{5}$&$2^{7}$ &557.54    &8.00      &557.99   &8.00       \\
         &$2^{6}$&$2^{8}$ &9109.38   &8.00      &9109.44  &8.00         \\
\hline
\end{tabular}
\end{table}

\begin{table}[!tbp]
\centering
\tabcolsep=5pt
\renewcommand{\arraystretch}{0.5}
\caption{Comparison of CPU time and iteration numbers for different $M$, $N$ and orders $(\alpha_1,\alpha_2,\alpha_3)=(1.5,1.7,1.9)$ and $(\alpha_1,\alpha_2,\alpha_3)=(1.1,1.5,1.9)$ in Example \ref{example2} at $T=1$.}
\label{table8}
\begin{tabular}{ccccccc}
\hline
\multirow{2}{*}{Methods}& \multirow{2}{*}{$M$}& \multirow{2}{*}{$N$}&\multicolumn{2}{c}{$(\alpha_1,\alpha_2,\alpha_3)=(1.5,1.7,1.9)$} &\multicolumn{2}{c}{$(\alpha_1,\alpha_2,\alpha_3)=(1.1,1.5,1.9)$}
   \\ \cmidrule(r){4-5} \cmidrule(r){6-7}
 &&  &CPU(s)  &Iter &CPU(s)  &Iter  \\
\hline 
\multirow{4}{*}{CG}        
        &$2^{4}$&$2^{6}$ &156.14    &255.50       &164.95    &238.00          \\
        &$2^{5}$&$2^{7}$ &4896.43   &434.00       &5244.91   &465.00       \\
        &$2^{6}$&$2^{8}$ &$>10000$  &$\dagger$    &$>10000$  &$\dagger$       \\
\hline          
 \multirow{4}{*}{P$_\tau$-CG}        
         &$2^{4}$&$2^{6}$ &14.03     &7.00      &15.24     &8.00          \\
         &$2^{5}$&$2^{7}$ &231.80    &8.00      &231.62    &8.00       \\
         &$2^{6}$&$2^{8}$ &3681.58   &8.00      &4064.83   &9.00         \\        
\hline          
 \multirow{4}{*}{P$_T$-CG}        
        &$2^{4}$&$2^{6}$ &39.75     &49.00      &42.04     &52.00          \\
        &$2^{5}$&$2^{7}$ &1086.96   &73.88      &1277.61   &83.00       \\
        &$2^{6}$&$2^{8}$ &$>10000$  &$\dagger$  &$>10000$  &$\dagger$         \\
\hline          
 \multirow{4}{*}{P$_S$-CG}        
        &$2^{4}$&$2^{6}$ &30.13    &35.19      &34.49     &40.94          \\
        &$2^{5}$&$2^{7}$ &729.09   &47.81      &848.78    &57.47       \\
        &$2^{6}$&$2^{8}$ &$>10000$ &$\dagger$  &$>10000$  &$\dagger$         \\
\hline          
 \multirow{4}{*}{P$_S$-GMRES}        
         &$2^{4}$&$2^{6}$ &44.65     &41.00      &49.68     &49.00          \\
         &$2^{5}$&$2^{7}$ &1066.00   &50.00      &1425.04   &65.53       \\
         &$2^{6}$&$2^{8}$ &$>10000$ &$\dagger$   &$>10000$  &$\dagger$          \\
\hline          
 \multirow{4}{*}{P$_{\tau_1}$-GMRES}        
        &$2^{4}$&$2^{6}$ &17.61     &8.00      &17.44     &8.00          \\
        &$2^{5}$&$2^{7}$ &272.33    &8.00      &300.46    &9.00       \\
        &$2^{6}$&$2^{8}$ &4439.35   &8.00      &4829.73   &9.00         \\
\hline          
 \multirow{4}{*}{P$_{\tau_2}$-GMRES}        
         &$2^{4}$&$2^{6}$ &17.57     &8.00      &17.49     &8.00          \\
         &$2^{5}$&$2^{7}$ &272.58    &8.00      &300.34    &9.00       \\
         &$2^{6}$&$2^{8}$ &4610.52   &8.58      &4875.44   &9.00         \\
\hline          
 \multirow{4}{*}{P$_{\tau_c}$-GMRES}        
         &$2^{4}$&$2^{6}$ &32.54     &7.00      &33.56    &7.00          \\
         &$2^{5}$&$2^{7}$ &509.99    &7.00      &509.69   &7.00       \\
         &$2^{6}$&$2^{8}$ &8173.98   &7.00      &8929.78  &8.00         \\
\hline
\end{tabular}
\end{table}

\begin{figure}[htbp]
	\centering
	\begin{minipage}{0.49\linewidth}   
		\centering
		\includegraphics[width=1.09\linewidth]{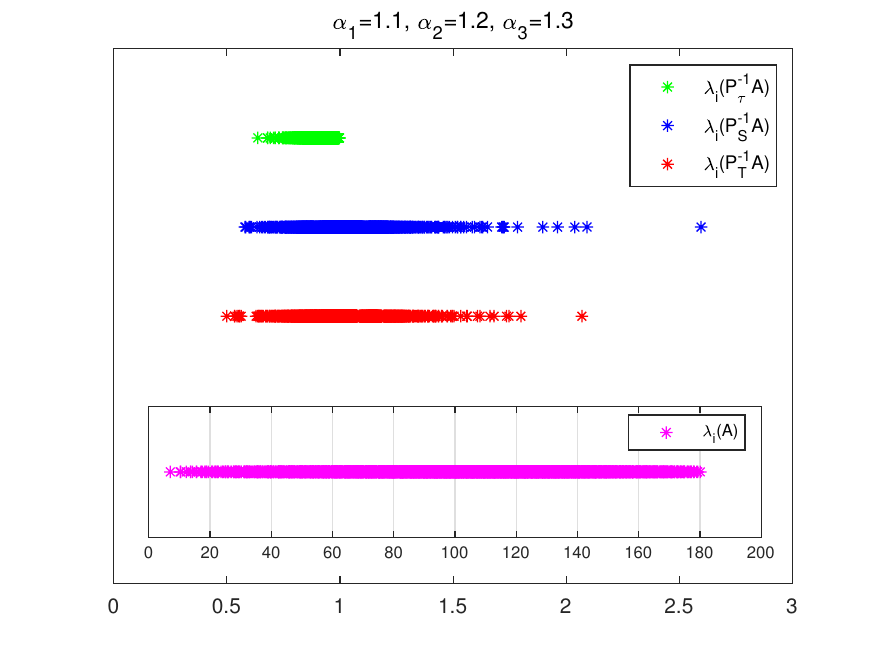}
	\end{minipage}
	\begin{minipage}{0.49\linewidth}
		\centering
		\includegraphics[width=1.09\linewidth]{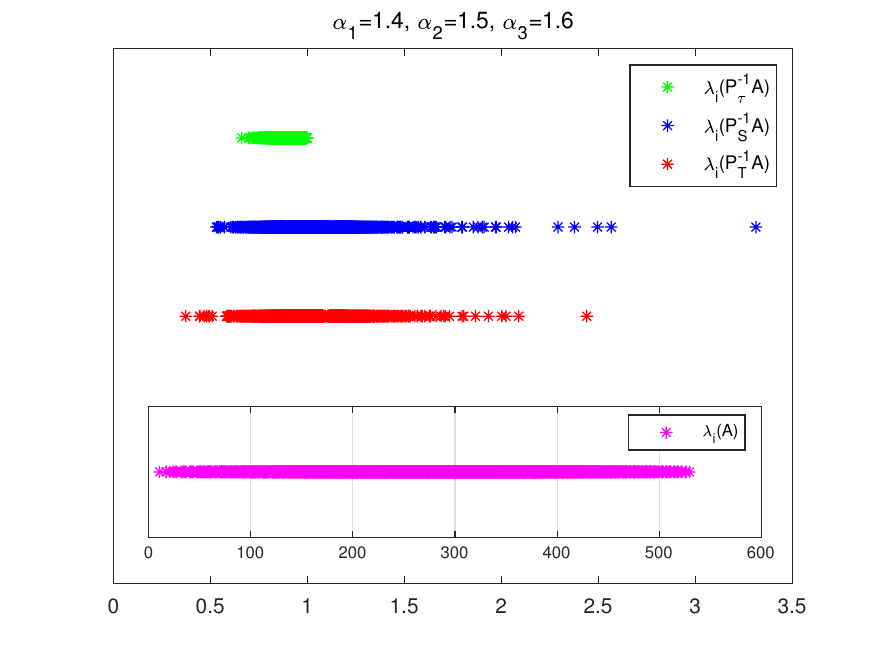}
	\end{minipage}
   \begin{minipage}{0.49\linewidth}   
		\centering
		\includegraphics[width=1.09\linewidth]{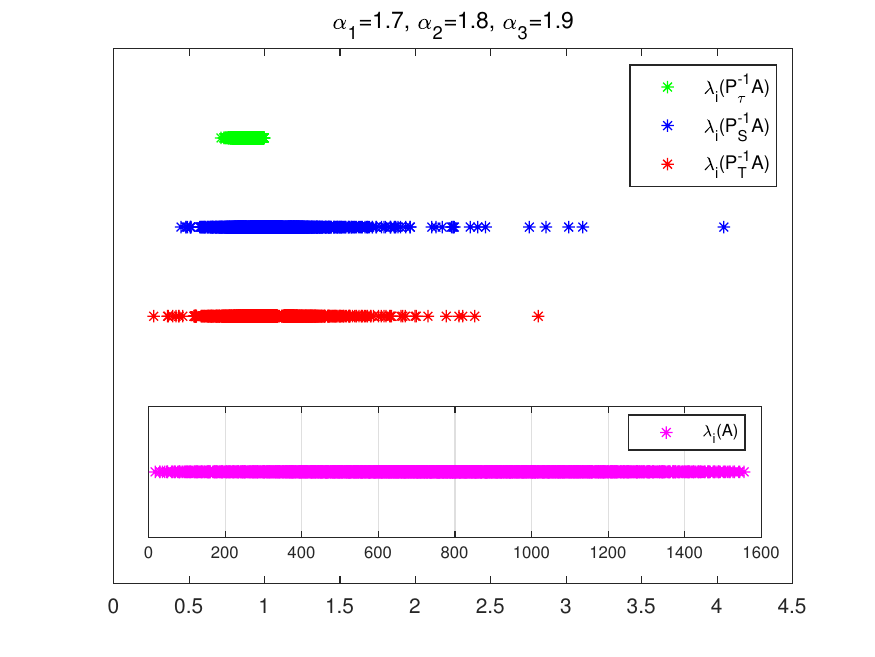}
	\end{minipage}
	\begin{minipage}{0.49\linewidth}
		\centering
		\includegraphics[width=1.09\linewidth]{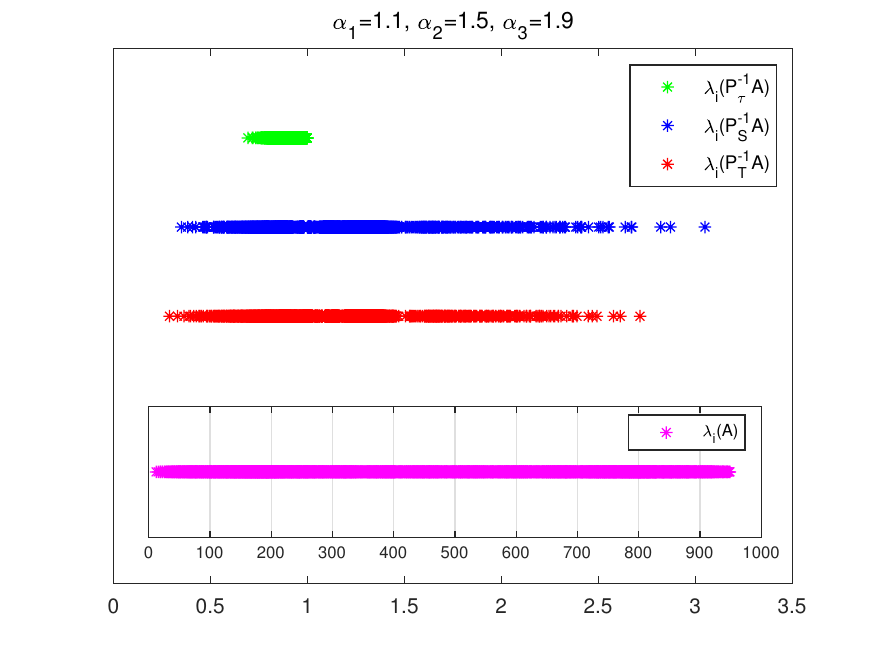}
	\end{minipage} 
\caption{The eigenvalue distribution of the matrices related to CG and PCG methods in Example \ref{example2} at the last time step with $n_1+1=n_2+1=n_3+1=2^4, M=2^6$ for difference fractional orders.}
\label{fig4}       
\end{figure}

\section{Conclusions}\label{sec:con}
In this paper, a novel CN-4FCD method is employed to discretize multi-dimensional Riesz space fractional diffusion equations with variable coefficients. The unconditional stability and convergence are analyzed in detail. Moreover, we propose the PCG method with sine transform based preconditioner for solving the resulting linear systems. Theoretically, we show that the spectra of the preconditioned matrices are uniformly bounded by two positive constants, which are away from zero and independent of the matrix size. This result indicates that the iteration numbers of the PCG method with the proposed sine transform based preconditioner are independent of the mesh size. Two numerical examples are given to support the accuracy of the scheme and the efficiency of the proposed preconditioner. It is worth noting that, compared to other existing preconditioners, the simple structure and ease of implementation of our proposed preconditioner can achieve not only optimal convergence, but also the shortest computational time. In our future work, the optimal preconditioned strategy for high-order scheme for multidimensional fractional diffusion equations with non-symmetric spatial operators, such as the Riemann-Liouville operator, will be studied.



\bibliographystyle{elsarticle-num-names}
\bibliography{ref}

\begin{thebibliography}{62}
\expandafter\ifx\csname natexlab\endcsname\relax\def\natexlab#1{#1}\fi
\providecommand{\url}[1]{\texttt{#1}}
\providecommand{\href}[2]{#2}
\providecommand{\path}[1]{#1}
\providecommand{\DOIprefix}{doi:}
\providecommand{\ArXivprefix}{arXiv:}
\providecommand{\URLprefix}{URL: }
\providecommand{\Pubmedprefix}{pmid:}
\providecommand{\doi}[1]{\href{http://dx.doi.org/#1}{\path{#1}}}
\providecommand{\Pubmed}[1]{\href{pmid:#1}{\path{#1}}}
\providecommand{\bibinfo}[2]{#2}
\ifx\xfnm\relax \def\xfnm[#1]{\unskip,\space#1}\fi
\bibitem[{Benson et~al.(2000{\natexlab{a}})Benson, Wheatcraft, and
  Meerschaert}]{BWM1}
\bibinfo{author}{D.~A. Benson}, \bibinfo{author}{S.~W. Wheatcraft},
  \bibinfo{author}{M.~M. Meerschaert},
\newblock \bibinfo{title}{Application of a fractional advection-dispersion
  equation},
\newblock \bibinfo{journal}{Water Resour. Res.} \bibinfo{volume}{36}
  (\bibinfo{year}{2000}{\natexlab{a}}) \bibinfo{pages}{1403--1412}.
\bibitem[{Benson et~al.(2000{\natexlab{b}})Benson, Wheatcraft, and
  Meerschaert}]{BWM2}
\bibinfo{author}{D.~A. Benson}, \bibinfo{author}{S.~W. Wheatcraft},
  \bibinfo{author}{M.~M. Meerschaert},
\newblock \bibinfo{title}{The fractional-order governing equation of {L{\'e}vy}
  motion},
\newblock \bibinfo{journal}{Water Resour. Res.} \bibinfo{volume}{36}
  (\bibinfo{year}{2000}{\natexlab{b}}) \bibinfo{pages}{1413--1423}.
\bibitem[{Magin(2006)}]{Magin1}
\bibinfo{author}{R.~L. Magin}, \bibinfo{title}{{Fractional Calculus in
  Bioengineering}}, \bibinfo{publisher}{Begell House Redding},
  \bibinfo{address}{New York}, \bibinfo{year}{2006}.
\bibitem[{Bai and Feng(2007)}]{BF1}
\bibinfo{author}{J.~Bai}, \bibinfo{author}{X.~C. Feng},
\newblock \bibinfo{title}{Fractional-order anisotropic diffusion for image
  denoising},
\newblock \bibinfo{journal}{IEEE Tran. Image Proc.} \bibinfo{volume}{16}
  (\bibinfo{year}{2007}) \bibinfo{pages}{2492--2502}.
\bibitem[{Raberto et~al.(2002)Raberto, Scalas, and Mainardi}]{RSM1}
\bibinfo{author}{M.~Raberto}, \bibinfo{author}{E.~Scalas},
  \bibinfo{author}{F.~Mainardi},
\newblock \bibinfo{title}{Waiting-times and returns in high-frequency financial
  data: an empirical study},
\newblock \bibinfo{journal}{Physica A} \bibinfo{volume}{314}
  (\bibinfo{year}{2002}) \bibinfo{pages}{749--755}.
\bibitem[{Meerschaert and Tadjeran(2004)}]{MT1}
\bibinfo{author}{M.~M. Meerschaert}, \bibinfo{author}{C.~Tadjeran},
\newblock \bibinfo{title}{Finite difference approximations for fractional
  advection--dispersion flow equations},
\newblock \bibinfo{journal}{J. Compt. Appl. Math.} \bibinfo{volume}{172}
  (\bibinfo{year}{2004}) \bibinfo{pages}{65--77}.
\bibitem[{Meerschaert and Tadjeran(2006)}]{MT2}
\bibinfo{author}{M.~M. Meerschaert}, \bibinfo{author}{C.~Tadjeran},
\newblock \bibinfo{title}{Finite difference approximations for two-sided
  space-fractional partial differential equations},
\newblock \bibinfo{journal}{Appl. Numer. Math.} \bibinfo{volume}{56}
  (\bibinfo{year}{2006}) \bibinfo{pages}{80--90}.
\bibitem[{Sousa and Li(2015)}]{SL1}
\bibinfo{author}{E.~Sousa}, \bibinfo{author}{C.~Li},
\newblock \bibinfo{title}{A weighted finite difference method for the
  fractional diffusion equation based on the {Riemann--Liouville} derivative},
\newblock \bibinfo{journal}{Appl. Numer. Math.} \bibinfo{volume}{90}
  (\bibinfo{year}{2015}) \bibinfo{pages}{22--37}.
\bibitem[{Tian et~al.(2015)Tian, Zhou, and Deng}]{TZD1}
\bibinfo{author}{W.~Y. Tian}, \bibinfo{author}{H.~Zhou}, \bibinfo{author}{W.~H.
  Deng},
\newblock \bibinfo{title}{A class of second order difference approximations for
  solving space fractional diffusion equations},
\newblock \bibinfo{journal}{Math. Comput.} \bibinfo{volume}{84}
  (\bibinfo{year}{2015}) \bibinfo{pages}{1703--1727}.
\bibitem[{Lin and She(2021)}]{lin2021stability}
\bibinfo{author}{F.~R. Lin}, \bibinfo{author}{Z.~H. She},
\newblock \bibinfo{title}{Stability and convergence of 3-point {WSGD} schemes
  for two-sided space fractional advection-diffusion equations with variable
  coefficients},
\newblock \bibinfo{journal}{Appl. Numer. Math.} \bibinfo{volume}{167}
  (\bibinfo{year}{2021}) \bibinfo{pages}{281--307}.
\bibitem[{Lin and Liu(2020)}]{LL1}
\bibinfo{author}{F.~R. Lin}, \bibinfo{author}{W.~D. Liu},
\newblock \bibinfo{title}{The accuracy and stability of {CN-WSGD} schemes for
  space fractional diffusion equation},
\newblock \bibinfo{journal}{J. Comput. Appl. Math.} \bibinfo{volume}{363}
  (\bibinfo{year}{2020}) \bibinfo{pages}{77--91}.
\bibitem[{She(2022)}]{she2022class}
\bibinfo{author}{Z.~H. She},
\newblock \bibinfo{title}{A class of unconditioned stable 4-point {WSGD}
  schemes and fast iteration methods for space fractional diffusion equations},
\newblock \bibinfo{journal}{J. Sci. Comput.} \bibinfo{volume}{92}
  (\bibinfo{year}{2022}) \bibinfo{pages}{1--35}.
\bibitem[{Zhu et~al.(2021)Zhu, Zhang, Fu, and Liu}]{zhu2021efficient}
\bibinfo{author}{C.~Zhu}, \bibinfo{author}{B.~Zhang}, \bibinfo{author}{H.~Fu},
  \bibinfo{author}{J.~Liu},
\newblock \bibinfo{title}{Efficient second-order {ADI} difference schemes for
  three-dimensional {R}iesz space-fractional diffusion equations},
\newblock \bibinfo{journal}{Comput. Math. Appl.} \bibinfo{volume}{98}
  (\bibinfo{year}{2021}) \bibinfo{pages}{24--39}.
\bibitem[{Huang et~al.(2023)Huang, Qu, and Lei}]{huang2023tau}
\bibinfo{author}{Y.~Y. Huang}, \bibinfo{author}{W.~Qu}, \bibinfo{author}{S.~L.
  Lei},
\newblock \bibinfo{title}{On $\tau$-preconditioner for a novel fourth-order
  difference scheme of two-dimensional {R}iesz space-fractional diffusion
  equations},
\newblock \bibinfo{journal}{Computers \& Mathematics with Applications}
  \bibinfo{volume}{145} (\bibinfo{year}{2023}) \bibinfo{pages}{124--140}.
\bibitem[{Fu et~al.(2019)Fu, Sun, Wang, and Zheng}]{fu2019stability}
\bibinfo{author}{H.~F. Fu}, \bibinfo{author}{Y.~N. Sun},
  \bibinfo{author}{H.~Wang}, \bibinfo{author}{X.~C. Zheng},
\newblock \bibinfo{title}{Stability and convergence of a {C}rank--{N}icolson
  finite volume method for space fractional diffusion equations},
\newblock \bibinfo{journal}{Appl. Numer. Math.} \bibinfo{volume}{139}
  (\bibinfo{year}{2019}) \bibinfo{pages}{38--51}.
\bibitem[{Donatelli et~al.(2018)Donatelli, Mazza, and
  Serra-Capizzano}]{donatelli2018spectral}
\bibinfo{author}{M.~Donatelli}, \bibinfo{author}{M.~Mazza},
  \bibinfo{author}{S.~Serra-Capizzano},
\newblock \bibinfo{title}{Spectral analysis and multigrid methods for finite
  volume approximations of space-fractional diffusion equations},
\newblock \bibinfo{journal}{SIAM J. Sci. Comput.} \bibinfo{volume}{40}
  (\bibinfo{year}{2018}) \bibinfo{pages}{A4007--A4039}.
\bibitem[{Liu et~al.(2014)Liu, Zhuang, Turner, Burrage, and Anh}]{LZTBA1}
\bibinfo{author}{F.~W. Liu}, \bibinfo{author}{P.~Zhuang},
  \bibinfo{author}{I.~Turner}, \bibinfo{author}{K.~Burrage},
  \bibinfo{author}{V.~Anh},
\newblock \bibinfo{title}{A new fractional finite volume method for solving the
  fractional diffusion equation},
\newblock \bibinfo{journal}{Appl. Math. Model.} \bibinfo{volume}{38}
  (\bibinfo{year}{2014}) \bibinfo{pages}{3871--3878}.
\bibitem[{Fu et~al.(2019)Fu, Liu, and Wang}]{fu2019finite}
\bibinfo{author}{H.~F. Fu}, \bibinfo{author}{H.~Liu},
  \bibinfo{author}{H.~Wang},
\newblock \bibinfo{title}{A finite volume method for two-dimensional
  {R}iemann-{L}iouville space-fractional diffusion equation and its efficient
  implementation},
\newblock \bibinfo{journal}{Journal of Computational Physics}
  \bibinfo{volume}{388} (\bibinfo{year}{2019}) \bibinfo{pages}{316--334}.
\bibitem[{Qu and Li(2021)}]{qu2021fast}
\bibinfo{author}{W.~Qu}, \bibinfo{author}{Z.~Li},
\newblock \bibinfo{title}{Fast direct solver for {CN-ADI-FV} scheme to
  two-dimensional {R}iesz space-fractional diffusion equations},
\newblock \bibinfo{journal}{Appl. Math. Comput.} \bibinfo{volume}{401}
  (\bibinfo{year}{2021}) \bibinfo{pages}{126033}.
\bibitem[{Jiang and Ma(2011)}]{JM1}
\bibinfo{author}{Y.~J. Jiang}, \bibinfo{author}{J.~T. Ma},
\newblock \bibinfo{title}{High-order finite element methods for time-fractional
  partial differential equations},
\newblock \bibinfo{journal}{J. Comput. Appl. Math.} \bibinfo{volume}{235}
  (\bibinfo{year}{2011}) \bibinfo{pages}{3285--3290}.
\bibitem[{Jin et~al.(2014)Jin, Lazarov, Pasciak, and Zhou}]{JLPZ1}
\bibinfo{author}{B.~T. Jin}, \bibinfo{author}{R.~Lazarov},
  \bibinfo{author}{J.~Pasciak}, \bibinfo{author}{Z.~Zhou},
\newblock \bibinfo{title}{Error analysis of a finite element method for the
  space-fractional parabolic equation},
\newblock \bibinfo{journal}{SIAM J. Numer. Anal.} \bibinfo{volume}{52}
  (\bibinfo{year}{2014}) \bibinfo{pages}{2272--2294}.
\bibitem[{Bu et~al.(2014)Bu, Tang, and Yang}]{bu2014galerkin}
\bibinfo{author}{W.~Bu}, \bibinfo{author}{Y.~Tang}, \bibinfo{author}{J.~Yang},
\newblock \bibinfo{title}{Galerkin finite element method for two-dimensional
  riesz space fractional diffusion equations},
\newblock \bibinfo{journal}{Journal of Computational Physics}
  \bibinfo{volume}{276} (\bibinfo{year}{2014}) \bibinfo{pages}{26--38}.
\bibitem[{Pang and Sun(2012)}]{PS1}
\bibinfo{author}{H.~K. Pang}, \bibinfo{author}{H.~W. Sun},
\newblock \bibinfo{title}{Multigrid method for fractional diffusion equations},
\newblock \bibinfo{journal}{J. Comput. Phys.} \bibinfo{volume}{231}
  (\bibinfo{year}{2012}) \bibinfo{pages}{693--703}.
\bibitem[{Moghaderi et~al.(2017)Moghaderi, Dehghan, Donatelli, and
  Mazza}]{moghaderi2017spectral}
\bibinfo{author}{H.~Moghaderi}, \bibinfo{author}{M.~Dehghan},
  \bibinfo{author}{M.~Donatelli}, \bibinfo{author}{M.~Mazza},
\newblock \bibinfo{title}{Spectral analysis and multigrid preconditioners for
  two-dimensional space-fractional diffusion equations},
\newblock \bibinfo{journal}{J. Comput. Phys.} \bibinfo{volume}{350}
  (\bibinfo{year}{2017}) \bibinfo{pages}{992--1011}.
\bibitem[{Lin et~al.(2017)Lin, Ng, and Sun}]{lin2017multigrid}
\bibinfo{author}{X.-l. Lin}, \bibinfo{author}{M.~K. Ng}, \bibinfo{author}{H.-W.
  Sun},
\newblock \bibinfo{title}{A multigrid method for linear systems arising from
  time-dependent two-dimensional space-fractional diffusion equations},
\newblock \bibinfo{journal}{Journal of Computational Physics}
  \bibinfo{volume}{336} (\bibinfo{year}{2017}) \bibinfo{pages}{69--86}.
\bibitem[{Lei and Sun(2013)}]{lei2013circulant}
\bibinfo{author}{S.~L. Lei}, \bibinfo{author}{H.~W. Sun},
\newblock \bibinfo{title}{A circulant preconditioner for fractional diffusion
  equations},
\newblock \bibinfo{journal}{J. Comput. Phys.} \bibinfo{volume}{242}
  (\bibinfo{year}{2013}) \bibinfo{pages}{715--725}.
\bibitem[{Lei et~al.(2016)Lei, Chen, and Zhang}]{lei2016multilevel}
\bibinfo{author}{S.~L. Lei}, \bibinfo{author}{X.~Chen}, \bibinfo{author}{X.~H.
  Zhang},
\newblock \bibinfo{title}{Multilevel circulant preconditioner for
  high-dimensional fractional diffusion equations},
\newblock \bibinfo{journal}{East Asian J. Appl. Math.} \bibinfo{volume}{6}
  (\bibinfo{year}{2016}) \bibinfo{pages}{109--130}.
\bibitem[{Jin et~al.(2015)Jin, Lin, and Zhao}]{JLZ1}
\bibinfo{author}{X.~Q. Jin}, \bibinfo{author}{F.~R. Lin},
  \bibinfo{author}{Z.~Zhao},
\newblock \bibinfo{title}{Preconditioned iterative methods for two-dimensional
  space-fractional diffusion equations},
\newblock \bibinfo{journal}{Commun. Comput. Phys.} \bibinfo{volume}{18}
  (\bibinfo{year}{2015}) \bibinfo{pages}{469--488}.
\bibitem[{She et~al.(2021)She, Lao, Yang, and Lin}]{SLYL1}
\bibinfo{author}{Z.~H. She}, \bibinfo{author}{C.~X. Lao},
  \bibinfo{author}{H.~Yang}, \bibinfo{author}{F.~R. Lin},
\newblock \bibinfo{title}{Banded preconditioners for {Riesz} space fractional
  diffusion equations},
\newblock \bibinfo{journal}{J. Sci. Comput.} \bibinfo{volume}{86}
  (\bibinfo{year}{2021}) \bibinfo{pages}{1--22}.
\bibitem[{Pan et~al.(2014)Pan, Ke, Ng, and Sun}]{PKNS1}
\bibinfo{author}{J.~Y. Pan}, \bibinfo{author}{R.~H. Ke}, \bibinfo{author}{M.~K.
  Ng}, \bibinfo{author}{H.~W. Sun},
\newblock \bibinfo{title}{Preconditioning techniques for
  diagonal-times-{Toeplitz} matrices in fractional diffusion equations},
\newblock \bibinfo{journal}{SIAM J. Sci. Comput.} \bibinfo{volume}{36}
  (\bibinfo{year}{2014}) \bibinfo{pages}{A2698--A2719}.
\bibitem[{Donatelli et~al.(2016)Donatelli, Mazza, and Serra-Capizzano}]{DMS1}
\bibinfo{author}{M.~Donatelli}, \bibinfo{author}{M.~Mazza},
  \bibinfo{author}{S.~Serra-Capizzano},
\newblock \bibinfo{title}{Spectral analysis and structure preserving
  preconditioners for fractional diffusion equations},
\newblock \bibinfo{journal}{J. Comput. Phys.} \bibinfo{volume}{307}
  (\bibinfo{year}{2016}) \bibinfo{pages}{262--279}.
\bibitem[{Aceto and Mazza(2023)}]{aceto2023rational}
\bibinfo{author}{L.~Aceto}, \bibinfo{author}{M.~Mazza},
\newblock \bibinfo{title}{A rational preconditioner for multi-dimensional riesz
  fractional diffusion equations},
\newblock \bibinfo{journal}{Computers \& Mathematics with Applications}
  \bibinfo{volume}{143} (\bibinfo{year}{2023}) \bibinfo{pages}{372--382}.
\bibitem[{Huang et~al.(2022)Huang, Lin, Ng, and Sun}]{huang2021spectral}
\bibinfo{author}{X.~Huang}, \bibinfo{author}{X.~L. Lin}, \bibinfo{author}{M.~K.
  Ng}, \bibinfo{author}{H.~W. Sun},
\newblock \bibinfo{title}{Spectral analysis for preconditioning of
  multi-dimensional {R}iesz fractional diffusion equations},
\newblock \bibinfo{journal}{Numer. Math. Theor. Meth. Appl.}
  \bibinfo{volume}{15} (\bibinfo{year}{2022}) \bibinfo{pages}{565--591}.
\bibitem[{Zhang et~al.(2022)Zhang, Yu, and Wang}]{zhang2022fast}
\bibinfo{author}{C.~H. Zhang}, \bibinfo{author}{J.~W. Yu},
  \bibinfo{author}{X.~Wang},
\newblock \bibinfo{title}{A fast second-order scheme for nonlinear {R}iesz
  space-fractional diffusion equations},
\newblock \bibinfo{journal}{Numer. Algor.}  (\bibinfo{year}{2022})
  \bibinfo{pages}{1--24}.
\bibitem[{Hon et~al.(2024)Hon, Li, Sormani, Krause, and
  Serra-Capizzano}]{hon2024symbol}
\bibinfo{author}{S.~Y. Hon}, \bibinfo{author}{C.~Li}, \bibinfo{author}{R.~L.
  Sormani}, \bibinfo{author}{R.~Krause}, \bibinfo{author}{S.~Serra-Capizzano},
\newblock \bibinfo{title}{Symbol-based multilevel block $\tau $ preconditioners
  for multilevel block toeplitz systems: Glt-based analysis and applications},
\newblock \bibinfo{journal}{arXiv preprint arXiv:2409.20363}
  (\bibinfo{year}{2024}).
\bibitem[{Huang et~al.(2024)Huang, Hon, Chou, and Lei}]{huang2024optimal}
\bibinfo{author}{Y.-Y. Huang}, \bibinfo{author}{S.~Y. Hon},
  \bibinfo{author}{L.-K. Chou}, \bibinfo{author}{S.-L. Lei},
\newblock \bibinfo{title}{Optimal preconditioners for nonsymmetric multilevel
  toeplitz systems with application to solving non-local evolutionary partial
  differential equations},
\newblock \bibinfo{journal}{arXiv preprint arXiv:2409.15770}
  (\bibinfo{year}{2024}).
\bibitem[{Li and Hon(2025)}]{li2025multilevel}
\bibinfo{author}{C.~Li}, \bibinfo{author}{S.~Hon},
\newblock \bibinfo{title}{Multilevel tau preconditioners for symmetrized
  multilevel toeplitz systems with applications to solving space fractional
  diffusion equations},
\newblock \bibinfo{journal}{SIAM Journal on Matrix Analysis and Applications}
  \bibinfo{volume}{46} (\bibinfo{year}{2025}) \bibinfo{pages}{487--508}.
\bibitem[{Huang et~al.(2025)Huang, Fung, Hon, and Lin}]{huang2025efficient}
\bibinfo{author}{Y.-Y. Huang}, \bibinfo{author}{P.~Y. Fung},
  \bibinfo{author}{S.~Y. Hon}, \bibinfo{author}{X.-L. Lin},
\newblock \bibinfo{title}{An efficient preconditioner for evolutionary partial
  differential equations with $\theta$-method in time discretization},
\newblock \bibinfo{journal}{Journal of Scientific Computing}
  \bibinfo{volume}{103} (\bibinfo{year}{2025}) \bibinfo{pages}{1--24}.
\bibitem[{Bai and Lu(2020)}]{bai2020fast}
\bibinfo{author}{Z.~Z. Bai}, \bibinfo{author}{K.~Y. Lu},
\newblock \bibinfo{title}{Fast matrix splitting preconditioners for higher
  dimensional spatial fractional diffusion equations},
\newblock \bibinfo{journal}{Journal of Computational Physics}
  \bibinfo{volume}{404} (\bibinfo{year}{2020}) \bibinfo{pages}{109117}.
\bibitem[{Zhang et~al.(2024)Zhang, Gu, Zhao, Li, and Gu}]{zhang2024two}
\bibinfo{author}{X.~Zhang}, \bibinfo{author}{X.~M. Gu}, \bibinfo{author}{Y.~L.
  Zhao}, \bibinfo{author}{H.~Li}, \bibinfo{author}{C.~Y. Gu},
\newblock \bibinfo{title}{Two fast and unconditionally stable finite difference
  methods for riesz fractional diffusion equations with variable coefficients},
\newblock \bibinfo{journal}{Applied Mathematics and Computation}
  \bibinfo{volume}{462} (\bibinfo{year}{2024}) \bibinfo{pages}{128335}.
\bibitem[{She et~al.(2024)She, Zhang, Gu, and Serra~Capizzano}]{she2024tau}
\bibinfo{author}{Z.~H. She}, \bibinfo{author}{X.~Zhang}, \bibinfo{author}{X.~M.
  Gu}, \bibinfo{author}{S.~Serra~Capizzano},
\newblock \bibinfo{title}{On $\tau$-preconditioners for a quasi-compact
  difference scheme to {R}iesz fractional diffusion equations with variable
  coefficients},
\newblock \bibinfo{journal}{arXiv preprint arXiv:2404.10221}
  (\bibinfo{year}{2024}).
\bibitem[{Bai et~al.(2017)Bai, Lu, and Pan}]{bai2017diagonal}
\bibinfo{author}{Z.~Z. Bai}, \bibinfo{author}{K.~Y. Lu}, \bibinfo{author}{J.~Y.
  Pan},
\newblock \bibinfo{title}{Diagonal and {T}oeplitz splitting iteration methods
  for diagonal-plus-{T}oeplitz linear systems from spatial fractional diffusion
  equations},
\newblock \bibinfo{journal}{Numerical Linear Algebra with Applications}
  \bibinfo{volume}{24} (\bibinfo{year}{2017}) \bibinfo{pages}{e2093}.
\bibitem[{Ng and Pan(2010)}]{ng2010approximate}
\bibinfo{author}{M.~K. Ng}, \bibinfo{author}{J.~Pan},
\newblock \bibinfo{title}{Approximate inverse circulant-plus-diagonal
  preconditioners for {T}oeplitz-plus-diagonal matrices},
\newblock \bibinfo{journal}{SIAM Journal on Scientific Computing}
  \bibinfo{volume}{32} (\bibinfo{year}{2010}) \bibinfo{pages}{1442--1464}.
\bibitem[{Lu(2018)}]{Lu-CAM}
\bibinfo{author}{K.~Y. Lu},
\newblock \bibinfo{title}{Diagonal and circulant or skew-circulant splitting
  preconditioners for spatial fractional diffusion equations},
\newblock \bibinfo{journal}{Comput. Appl. Math.} \bibinfo{volume}{37}
  (\bibinfo{year}{2018}) \bibinfo{pages}{4196--4218}.
\bibitem[{Bini and Benedetto(1990)}]{bini1990new}
\bibinfo{author}{D.~Bini}, \bibinfo{author}{F.~Benedetto},
\newblock \bibinfo{title}{A new preconditioner for the parallel solution of
  positive definite toeplitz systems},
\newblock in: \bibinfo{booktitle}{Proceedings of the second annual ACM
  Symposium on Parallel Algorithms and Architectures}, \bibinfo{year}{1990},
  pp. \bibinfo{pages}{220--223}.
\bibitem[{Serra(1999)}]{serra1999superlinear}
\bibinfo{author}{S.~Serra},
\newblock \bibinfo{title}{Superlinear {PCG} methods for symmetric {T}oeplitz
  systems},
\newblock \bibinfo{journal}{Math. Comput.} \bibinfo{volume}{68}
  (\bibinfo{year}{1999}) \bibinfo{pages}{793--803}.
\bibitem[{Lu et~al.(2021)Lu, Fang, and Sun}]{lu2021splitting}
\bibinfo{author}{X.~Lu}, \bibinfo{author}{Z.~W. Fang}, \bibinfo{author}{H.~W.
  Sun},
\newblock \bibinfo{title}{Splitting preconditioning based on sine transform for
  time-dependent {R}iesz space fractional diffusion equations},
\newblock \bibinfo{journal}{J. Appl. Math. Comput.} \bibinfo{volume}{66}
  (\bibinfo{year}{2021}) \bibinfo{pages}{673--700}.
\bibitem[{Tang and Huang(2022)}]{Tang2022AML}
\bibinfo{author}{S.~P. Tang}, \bibinfo{author}{Y.~M. Huang},
\newblock \bibinfo{title}{A lopsided scaled {DTS} preconditioning method for
  the discrete space-fractional diffusion equations},
\newblock \bibinfo{journal}{Appl. Math. Lett.} \bibinfo{volume}{131}
  (\bibinfo{year}{2022}) \bibinfo{pages}{108022}.
\bibitem[{Shao and Kang(2022)}]{shao2022preconditioner}
\bibinfo{author}{X.~H. Shao}, \bibinfo{author}{C.~B. Kang},
\newblock \bibinfo{title}{A preconditioner based on sine transform for space
  fractional diffusion equations},
\newblock \bibinfo{journal}{Appl. Numer. Math.} \bibinfo{volume}{178}
  (\bibinfo{year}{2022}) \bibinfo{pages}{248--261}.
\bibitem[{Tang and Huang(2024)}]{tang2024new}
\bibinfo{author}{S.~P. Tang}, \bibinfo{author}{Y.~M. Huang},
\newblock \bibinfo{title}{A new diagonal and {T}oeplitz splitting
  preconditioning method for solving time-dependent {R}iesz space-fractional
  diffusion equations},
\newblock \bibinfo{journal}{Applied Mathematics Letters} \bibinfo{volume}{149}
  (\bibinfo{year}{2024}) \bibinfo{pages}{108901}.
\bibitem[{Zeng et~al.(2022)Zeng, Yang, and Zhang}]{zeng2022tau}
\bibinfo{author}{M.~L. Zeng}, \bibinfo{author}{J.~F. Yang},
  \bibinfo{author}{G.~F. Zhang},
\newblock \bibinfo{title}{On $\tau$ matrix-based approximate inverse
  preconditioning technique for diagonal-plus-{T}oeplitz linear systems from
  spatial fractional diffusion equations},
\newblock \bibinfo{journal}{J. Compt. Appl. Math.}  (\bibinfo{year}{2022})
  \bibinfo{pages}{114088}.
\bibitem[{Greenbaum et~al.(1996)Greenbaum, Pt{\'a}k, and
  Strako{\v{s}}}]{greenbaum1996any}
\bibinfo{author}{A.~Greenbaum}, \bibinfo{author}{V.~Pt{\'a}k},
  \bibinfo{author}{Z.~e.~k. Strako{\v{s}}},
\newblock \bibinfo{title}{Any nonincreasing convergence curve is possible for
  gmres},
\newblock \bibinfo{journal}{Siam journal on matrix analysis and applications}
  \bibinfo{volume}{17} (\bibinfo{year}{1996}) \bibinfo{pages}{465--469}.
\bibitem[{Greenbaum and Strakos(1994)}]{greenbaum1994matrices}
\bibinfo{author}{A.~Greenbaum}, \bibinfo{author}{Z.~Strakos},
  \bibinfo{title}{Matrices that generate the same Krylov residual spaces},
  \bibinfo{publisher}{Springer}, \bibinfo{year}{1994}.
\bibitem[{Ortigueira(2006)}]{ortigueira2006riesz}
\bibinfo{author}{M.~D. Ortigueira},
\newblock \bibinfo{title}{Riesz potential operators and inverses via fractional
  centred derivatives},
\newblock \bibinfo{journal}{Int. J. Math. Math. Sci.} \bibinfo{volume}{2006}
  (\bibinfo{year}{2006}).
\bibitem[{\c{C}elik and Duman(2012)}]{CD2}
\bibinfo{author}{C.~\c{C}elik}, \bibinfo{author}{M.~Duman},
\newblock \bibinfo{title}{{Crank-Nicolson} method for the fractional diffusion
  equation with the {Riesz} fractional derivative},
\newblock \bibinfo{journal}{J. Comput. Phys.} \bibinfo{volume}{231}
  (\bibinfo{year}{2012}) \bibinfo{pages}{1743--1750}.
\bibitem[{Zhang et~al.(2025)Zhang, Gu, and Zhao}]{zhang2025two}
\bibinfo{author}{X.~Zhang}, \bibinfo{author}{X.~M. Gu}, \bibinfo{author}{Y.~L.
  Zhao},
\newblock \bibinfo{title}{Two fast finite difference methods for a class of
  variable-coefficient fractional diffusion equations with time delay},
\newblock \bibinfo{journal}{Communications in Nonlinear Science and Numerical
  Simulation} \bibinfo{volume}{140} (\bibinfo{year}{2025})
  \bibinfo{pages}{108358}.
\bibitem[{Hao et~al.(2015)Hao, Sun, and Cao}]{hao2015fourth}
\bibinfo{author}{Z.~P. Hao}, \bibinfo{author}{Z.~Z. Sun},
  \bibinfo{author}{W.~R. Cao},
\newblock \bibinfo{title}{A fourth-order approximation of fractional
  derivatives with its applications},
\newblock \bibinfo{journal}{Journal of Computational Physics}
  \bibinfo{volume}{281} (\bibinfo{year}{2015}) \bibinfo{pages}{787--805}.
\bibitem[{Ding and Li(2023)}]{ding2023high}
\bibinfo{author}{H.~F. Ding}, \bibinfo{author}{C.~P. Li},
\newblock \bibinfo{title}{High-order numerical algorithm and error analysis for
  the two-dimensional nonlinear spatial fractional complex {G}inzburg--{L}andau
  equation},
\newblock \bibinfo{journal}{Communications in Nonlinear Science and Numerical
  Simulation} \bibinfo{volume}{120} (\bibinfo{year}{2023})
  \bibinfo{pages}{107160}.
\bibitem[{Chan and Jin(2007)}]{chan2007introduction}
\bibinfo{author}{R.~H.~F. Chan}, \bibinfo{author}{X.~Q. Jin},
  \bibinfo{title}{An introduction to iterative Toeplitz solvers},
  \bibinfo{publisher}{SIAM}, \bibinfo{year}{2007}.
\bibitem[{Qu et~al.(2025)Qu, Huang, Hon, and Lei}]{qu2025novel}
\bibinfo{author}{W.~Qu}, \bibinfo{author}{Y.~Y. Huang},
  \bibinfo{author}{S.~Hon}, \bibinfo{author}{S.~L. Lei},
\newblock \bibinfo{title}{A novel fourth-order scheme for two-dimensional
  {R}iesz space fractional nonlinear reaction-diffusion equations and its
  optimal preconditioned solver},
\newblock \bibinfo{journal}{Numerical Linear Algebra with Applications}
  \bibinfo{volume}{32} (\bibinfo{year}{2025}) \bibinfo{pages}{e70005}.
\bibitem[{Lin et~al.(2018)Lin, Ng, and Sun}]{lin2018efficient}
\bibinfo{author}{X.~L. Lin}, \bibinfo{author}{M.~K. Ng}, \bibinfo{author}{H.~W.
  Sun},
\newblock \bibinfo{title}{Efficient preconditioner of one-sided space
  fractional diffusion equation},
\newblock \bibinfo{journal}{BIT Numerical Mathematics} \bibinfo{volume}{58}
  (\bibinfo{year}{2018}) \bibinfo{pages}{729--748}.
\bibitem[{Greenbaum(1997)}]{greenbaum1997iterative}
\bibinfo{author}{A.~Greenbaum}, \bibinfo{title}{Iterative methods for solving
  linear systems}, \bibinfo{publisher}{SIAM}, \bibinfo{year}{1997}.

\end{thebibliography}

\end{document}